\documentclass[12pt,a4paper]{amsart}
\usepackage{graphics}
\usepackage{amscd}
\usepackage{color}
\usepackage{comment}
\usepackage{oldgerm}
\usepackage{amsmath}
\usepackage{amssymb}
\usepackage{amsthm}
\usepackage{bm}
\newtheorem{thm}{Theorem}[section]
\newtheorem{prop}[thm]{Proposition}
\newtheorem{lem}[thm]{Lemma}

\newtheorem{definition}[thm]{Definition}
\newtheorem{rem}[thm]{Remark}

\newcommand{\quater}{\mathbb{H}}
\newcommand{\Z}{\mathbb{Z}}

\newcommand{\R}{\mathbb{R}}
\newcommand{\C}{\mathbb{C}}
\newcommand{\T}{\mathbb{T}}

\newcommand{\E}{\mathcal{E}}
\newcommand{\F}{\mathcal{F}}
\newcommand{\bdomega}{\boldsymbol{\omega}}
\newcommand{\bdeta}{\boldsymbol{\eta}}
\newcommand{\vol}{{\rm vol}}

\newcommand{\del}{\partial}

\title[The energy of maps]{
The energy of maps accompanying the collapsing of the $K3$ surface to a flat 3-dimensional orbifold}
\author[K. Hattori]{Kota Hattori}
\address{Keio University, 
3-14-1 Hiyoshi, Kohoku, Yokohama 223-8522, Japan}
\email{hattori@math.keio.ac.jp}

\begin{document}
\maketitle

\begin{abstract}
We study the Dirichlet energy of some smooth maps appearing in a collapsing family of hyper-K\"ahler metrics on the $K3$ surface constructed by Foscolo. We introduce an invariant for homotopy classes of smooth maps from the $K3$ surface with a hyper-K\"ahler metric to a flat Riemannian orbifold of dimension $3$, then show that it gives a lower bound of the energy. Moreover, we show that the ratio of the energy to the invariant converges to $1$ for Foscolo's collapsing families. 
\end{abstract}

\section{Introduction}
From the viewpoint of differential geometry, one of the significant problems on the $K3$ surface is to classify and construct the collapsing of hyper-K\"ahler metrics. On the compact complex surfaces, the hyper-K\"ahler metrics are equivalent to the Ricci-flat K\"ahler metrics with holomorphic volume form. 
By Yau's Theorem \cite{Yau1978} for the Calabi conjecture, the Ricci-flat K\"ahler metrics always exist on compact K\"ahler surfaces with trivial canonical bundles, hence hyper-K\"ahler metrics exist on $K3$ surfaces. 
In \cite{GW2000}, Gross and Wilson have constructed families of Ricci-flat K\"ahler metrics on elliptic $K3$ surfaces converging to $2$-sphere with some singular Riemannian metrics in the sense of Gromov-Hausdorff topology. 
Recently, Foscolo constructed families of hyper-K\"ahler metrics on the $K3$ surface converging to $3$-dimensional flat Riemannian orbifolds in \cite{Foscolo2019} and Hein, Sun, Viaclovsky and Zhang constructed families of hyper-K\"ahler metrics on the $K3$ surface converging to the $1$-dimensional segments in \cite{HSVZ2022}. 
Moreover, Sun and Zhang \cite{SZ2022} classified the limit space appearing as the Gromov-Hausdorff convergence of hyper-K\"ahler metrics on the $K3$ surface with the fixed diameters, and Odaka and Oshima \cite{OO2021} gave the conjecture of a more precise picture of collapsing, using the several notions of compactifications of symmetric spaces. 

In those collapsing families, we can observe that some maps from the $K3$ surface to the limit space appear naturally. 
The typical case is Gross and Wilson's example. In this case, we consider an elliptic fibration $f\colon X\to \C\mathbb{P}^1$ with $24$ singular fibers of Kodaira type $I_1$. 
Then the collapsing families of Ricci-flat K\"ahler metrics converge to $\C\mathbb{P}^1$ with all of the fibers shrinking. Here, $f$ can be regarded as a holomorphic map, or a special Lagrangian fibration after rotating hyper-K\"ahler structures. 

In differential geometry, holomorphic maps are typical examples of harmonic maps. Let $(M,g_M)$ and $(N,g_N)$ be Riemannian manifolds and suppose that $M$ is compact. For a smooth map $f\colon M\to N$ and a point $x\in M$, let $A$ be the representation matrix of a linear map $df_x\colon T_xM\to T_{f(x)}N$ for some orthonormal basis of $T_xM$ and $T_{f(x)}N$. 
We define a function $\| df\|_{g_M,g_N}\colon M\to \R$ by 
\begin{align*}
\| df\|_{g_M,g_N}(x):=\sqrt{ {\rm tr}({}^t\!AA)}. 
\end{align*}
Then the Dirichlet's energy of $f$ is defined by 
\begin{align*}
\E(f,g_M,g_N):=\int_M\| df\|_{g_M,g_N}^2d\mu_{g_M},
\end{align*}
where $\mu_{g_M}$ is the volume measure of $g_M$. 
Then $f$ is said to be harmonic if it is the critical point of $\E$. 
If $M,N$ are complex manifolds, $g_M,g_N$ are K\"ahler metrics and $f$ is holomorphic, then $f$ is not only harmonic but also minimizes $\E$ in its homotopy class by the following reason. 
Let $\omega_M,\omega_N$ be the K\"ahler form on $M,N$, respectively, and put $m=\dim_\C M$, $n=\dim_\C N$. 
Then, for any smooth map $f$, $\omega_M^{m-1}\wedge f^*\omega_N$ is $2m$-form on $M$, hence there is a function $\tau_f$ such that $\omega_M^{m-1}\wedge f^*\omega_N=\tau_f\vol_{g_M}$, where $\vol_{g_M}=\omega_M^m/m!$ is the volume form. 
Lichnerowicz \cite{Lichnerowicz1969} showed the inequality 
\begin{align*}
m\int_M\omega_M^{m-1}\wedge f^*\omega_N \le \E(f,g_M,g_N)
\end{align*}
and showed that the equality holds iff $f$ is holomorphic. Since the left-hand side is determined by the homotopy class of $f$, 
therefore, the holomorphic maps give the minimum of $\E$ restricted to their homotopy classes. 
The author generalized this observation to smooth maps between manifolds with the geometric structure characterized by differential forms in \cite{hattori2024calibrated}. 

In the case of the collapsing families constructed in \cite{Foscolo2019} and \cite{HSVZ2022}, we also need the approximation maps to show the Gromov-Hausdorff convergence, however, the differential geometric properties of these maps are not known well. 
In this paper, we study the asymptotic behavior of the energy of these maps in the case of collapsing families constructed by Foscolo in \cite{Foscolo2019}. Let $X$ be the smooth $4$-manifold diffeomorphic to the $K3$ surface and $\T:=T^3=\R^3/\Z^3$ be the $3$-dimensional torus with a flat Riemannian metric $g_{\T}$. Here, we have the standard $\{ \pm 1\}$-action on $T^3$ defined by $x\mapsto -x$, which has $8$ fixed points. 
Since the metric $g_{\T}$ is $\{ \pm 1\}$-invariant, it descends to  the metric on the orbifold $\T/\{ \pm 1\}$ which we also write $g_{\T}$. 
In this paper, we introduce an invariant for homotopy classes of smooth maps $X\to \T/\{ \pm 1\}$, and show that it gives the lower bound for the Dirichlet energy restricted to the fixed homotopy class.
Here, the smoothness of maps between orbifolds is defined in \cite{Satake1956}. For a smooth map $f\colon X\to \T/\{ \pm 1\}$, we denote by $[f]$ the homotopy class represented by $f$. 
Let $[X,\T/\{ \pm 1\}]$ be the set consisting of all homotopy classes of smooth maps from $X$ to $\T/\{ \pm 1\}$. 
Then we obtain the next main result. 
\begin{thm}
Let $\varepsilon_0>0$ be a sufficiently small positive constant and $(g_\varepsilon)_{\varepsilon\in(0,\varepsilon_0]}$ be a family of hyper-K\"ahler metrics constructed in \cite{Foscolo2019} such that 
$\{ (X,g_\varepsilon)\}_\varepsilon$ converges to $(\T/\{ \pm 1\}, g_{\T})$ as $\varepsilon\to 0$ in the sense of Gromov-Hausdorff topology. Then there is a function 
\begin{align*}
\mathcal{I}_\varepsilon\colon [X,\T/\{ \pm 1\}]\to \R
\end{align*}
such that $\E(f)\ge \mathcal{I}_\varepsilon([f])$ for any $\varepsilon$ and any smooth map $f\colon X\to \T/\{ \pm 1\}$. 
Moreover, there is a smooth map $F_\varepsilon \colon X\to \T/\{ \pm 1\}$ for each $\varepsilon$ satisfying the following properties. 
\begin{itemize}
\setlength{\parskip}{0cm}
\setlength{\itemsep}{0cm}
 \item[$({\rm i})$] $F_\varepsilon$ are approximation maps which give the Gromov-Hausdorff convergence $(X,g_\varepsilon)\to (\T/\{ \pm 1\},g_\T)$ as $\varepsilon\to 0$. 
 \item[$({\rm ii})$] $[F_\varepsilon]=[F_{\varepsilon'}]$ for all $\varepsilon,\varepsilon'$. 
 \item[$({\rm iii})$] $\mathcal{I}_\varepsilon([F_\varepsilon])>0$ for all $\varepsilon$ and $\lim_{\varepsilon\to 0}\mathcal{I}_\varepsilon([F_\varepsilon])=0$. 
 \item[$({\rm iv})$] $\lim_{\varepsilon\to 0}\E(F_\varepsilon)/\mathcal{I}_\varepsilon([F_\varepsilon])=1$. 
\end{itemize}
\label{thm main intro}
\end{thm}

This paper is organized as follows. 
In Section \ref{sec HK triple}, we review the definition and the fundamentals of hyper-K\"ahler manifolds of dimension $4$. 
In Section \ref{sec Foscolo}, we review the Gibbons-Hawking ansatz and the construction of collapsing families of hyper-K\"ahler metrics constructed in \cite{Foscolo2019}. 
Then we consider smooth maps from the $K3$ surface to $\T/\{ \pm 1\}$ and construct the invariants for homotopy classes, and construct $F_\varepsilon$ in Section \ref{sec map}, 
and show that the invariants give the lower bound for the Dirichlet energy in Section \ref{sec cal}. 
By the results of \cite{Foscolo2019}, $F_\varepsilon$ are 
approximation maps of the Gromov-Hausdorff convergence $(X,g_\varepsilon)\to (\T/\{ \pm 1\},g_\T)$. 
Moreover, in Section \ref{sec main estimate}, we give some estimates of the energy of $F_\varepsilon$ and show $({\rm ii, iii})$ of Theorem \ref{thm main intro}.

\noindent
{\bf Acknowledgement.}
This work was supported by JSPS KAKENHI Grant Numbers 19K03474, 23K20210, 24K06717 and 24H00183.

\section{Hyper-K\"ahler triples}\label{sec HK triple}
We define hyper-K\"ahler triples on smooth manifolds of dimension $4$. We follow the definition by \cite{Foscolo2019}. 
\begin{definition}
\normalfont
Let $V$ be a $4$-dimensional real vector space 
and $\boldsymbol{\omega}=(\omega_1,\omega_2,\omega_3)\in\Lambda^2V^*\otimes\R^3$. 
Then $\boldsymbol{\omega}$ is an $SU(2)$-{\it structure} if 
\begin{align*}
\omega_1^2&=\omega_2^2=\omega_3^2\neq 0,\\
\omega_1\wedge\omega_2
&=\omega_2\wedge\omega_3
=\omega_3\wedge\omega_1=0.
\end{align*}
\end{definition}
If $\bdomega$ is an $SU(2)$-structure on $V$, then 
we can define a symmetric tensor $S_{\boldsymbol{\omega}}\colon V\times V\to \R$ by 
\begin{align*}
\frac{1}{2}S_{\boldsymbol{\omega}}(u,v)\omega_1^2:=\iota_u\omega_1\wedge\iota_v\omega_2\wedge\omega_3.
\end{align*}
We can see that $S_{\boldsymbol{\omega}}$ 
is positive or negative definite, then we define a positive definite symmetric tensor $g_{\bdomega}\colon V\times V\to \R$ by 
\[ g_{\bdomega}:=
\left\{
\begin{array}{cl}
S_{\boldsymbol{\omega}} & (\mbox{ if }S_{\boldsymbol{\omega}}\mbox{ is positive}), \\
-S_{\boldsymbol{\omega}} & (\mbox{ if }S_{\boldsymbol{\omega}}\mbox{ is negative}).
\end{array}
\right.
\]

\begin{rem}
\normalfont
If $\bdomega$ is an $SU(2)$-structure, then $-\bdomega$ is also  an $SU(2)$-structure and $S_{-\boldsymbol{\omega}}=-S_{\boldsymbol{\omega}}$. 
\end{rem}
\begin{rem}
\normalfont
In general, $\bdomega$ is an $SU(2)$-structure with $S_{\bdomega}>0$ iff 
there is an orthonormal basis $\{ e_0,e_1,e_2,e_3\}\subset V$ for  $g_{\bdomega}$ such that 
$\omega_1=e^0\wedge e^1+e^2\wedge e^3$, 
$\omega_2=e^0\wedge e^2+e^3\wedge e^1$, 
$\omega_3=e^0\wedge e^3+e^1\wedge e^2$, 
where $\{ e^0,e^1,e^2,e^3\}\subset V^*$ is the dual basis. 
\end{rem}
\begin{rem}
\normalfont
If $\boldsymbol{\omega}$ is an $SU(2)$-structure, 
then $\omega_1^2$ induces an orientation on $V$. 
Then the volume form $\vol_{g_{\bdomega}}$ is given by $\omega_1^2/2$ and the space of self-dual $2$ forms $\Lambda_+^2V^*$ is generated by $\omega_1,\omega_2,\omega_3$. 
\end{rem}
\begin{definition}
\normalfont
$\boldsymbol{\omega}=(\omega_1,\omega_2,\omega_3)\in\Lambda^2V^*\otimes\R^3$ is a {\it definite triple} if there is a positive definite symmetric matrix $Q_{\bdomega}=(Q_{ij})_{i,j=1,2,3}$ and an orientation $\mu_{\boldsymbol{\omega}}\in\Lambda^4V^*\setminus\{ 0\}$ such that 
\begin{align*}
\omega_i\wedge\omega_j&=Q_{ij}\mu_{\boldsymbol{\omega}},\\
\det(Q_{\bdomega})&=1.
\end{align*}
Then there is an $SU(2)$-structure 
$\boldsymbol{\omega}'=(\omega'_1,\omega'_2,\omega'_3)$ such that ${\rm span}\{ \omega_1,\omega_2,\omega_3\}={\rm span}\{ \omega'_1,\omega'_2,\omega'_3\}$. 
Then we obtain a metric $g_{\boldsymbol{\omega}'}$. 
Here, we put $g_{\boldsymbol{\omega}}:=g_{\boldsymbol{\omega}'}$, which is independent of the choice of $\boldsymbol{\omega}'$. 
\end{definition}
\begin{rem}
\normalfont
If $\boldsymbol{\omega}$ is a definite triple, 
then ${\rm span}\{ \omega_1,\omega_2,\omega_3\}$ is 
the space of self-dual $2$-forms and the volume form of 
$g_{\boldsymbol{\omega}}$ is $\mu_{\boldsymbol{\omega}}/2$. 
It is obvious that $SU(2)$-structures are 
definite triples. 
\end{rem}
\begin{definition}
\normalfont
Let $X$ be a smooth manifold of dimension $4$ 
and $\boldsymbol{\omega}=(\omega_1,\omega_2,\omega_3)\in\Omega^2(X)\otimes\R^3$. 
Then $\boldsymbol{\omega}$ is a {\it hyper-K\"ahler triple} on $X$ if $\boldsymbol{\omega}_x$ is an $SU(2)$-structure on $T_xX$ for all $x\in X$ and 
\begin{align*}
d\omega_1&=d\omega_2=d\omega_3 = 0.
\end{align*}
Moreover, 
the induced Riemannian metric 
$g_{\boldsymbol{\omega}}$ is called a {\it hyper-K\"ahler metric}. 
Similarly, $\boldsymbol{\omega}$ is said to be a {\it definite triple} if $\boldsymbol{\omega}_x$ is a definite triple on $T_xX$ for all $x\in X$. 
\end{definition}

Here, we show some elementary estimates for metrics for the following sections. 
Recall that an inner product $g$ on $V$ induces 
the inner product on $\Lambda^k V^*$ naturally. 
We denote by $|\cdot |_g$ the induced norm with respect to $g$. 
For a triple of $k$-form $\bdeta=(\eta_1,\eta_2,\eta_3)\in\Lambda^kV^*\otimes \R^3$, 
we put $|\bdeta|_g:=\max_i|\eta_i|_g$. 

It is easy to show the next lemma. 
\begin{lem}
Let $V$ be a vector space and $g$ be a positive definite inner product on $V$. There is a positive constant $C$ depending only on $\dim (V)$ such that $|\alpha \wedge\beta|_g\le C|\alpha|_g|\beta|_g$ and $|\iota_u\alpha|_g\le C|u|_g|\alpha|_g$ for any $u\in V$, $\alpha\in\Lambda^k$ and $\beta\in\Lambda^l$. 
\label{lem norm asso}
\end{lem}

For $A\in M_3(\R)$, put 
\begin{align*}
|A|:=\max\{ |\lambda|\in\R|\, \lambda\mbox{ is an eigenvalue of }A\}.
\end{align*}
This is so-called the operator norm and we always have $|AB|=|A||B|$ for $A,B\in M_3(\R)$. 
For symmetric tensors $g_0,g_1$, we write $g_0\le g_1$ if $g_0(u,u)\le g_1(u,u)$ for any $u$. 

\begin{lem}
Let $V$ be a vector space of dimension $4$, 
$\bdomega$ be a definite triple, $\tilde{\bdomega}$ be an $SU(2)$-structure. 
There are a sufficiently large constant $C>0$ and a sufficiently small constant $\hat{\delta}>0$ such that the following hold.  
For any $\delta_1,\delta_2\in (0,\hat{\delta}]$, if $|\bdomega-\tilde{\bdomega}|_{g_{\bdomega}}\le \delta_1$ and $|Q_{\bdomega}-I_3|\le\delta_2$, where $I_3$ is the identity matrix in $M_3(\R)$, 
then we have $(1-C(\delta_1+\delta_2))g_{\bdomega}\le g_{\tilde{\bdomega}} \le (1+C(\delta_1+\delta_2))g_{\bdomega}$.
Similarly, if $|\bdomega-\tilde{\bdomega}|_{g_{\tilde{\bdomega}}}\le \delta_1$, then we have $(1-C\delta_1)g_{\bdomega}\le g_{\tilde{\bdomega}} \le (1+C\delta_1)g_{\bdomega}$. 
\label{lem norm met}
\end{lem}
\begin{proof}
Let $P$ be the positive symmetric square root of $Q_{\bdomega}$. Since $P+I_3$ is positive and $|(P+I_3)^{-1}|\le 1$, we have 
\begin{align*}
|P-I_3|=|(P+I_3)^{-1}||Q-I_3|
\le \delta_2.
\end{align*}
If we put $P^{-1}=(P^{ij})$ and $\bdomega':=P^{-1}\bdomega=(\sum_{j=1}^3P^{ij}\omega_j)_{i=1,2,3}$, 
then $Q_{\bdomega'}=I_3$, hence $\bdomega'$ is an $SU(2)$-structure. 
Then the metric $g_{\bdomega}$ is given by 
$S_{\bdomega'}$ or $-S_{\bdomega'}$. 
We have 
\begin{align*}
|\bdomega'-\tilde{\bdomega}|_{g_{\bdomega}}
&=|P^{-1}\bdomega-\tilde{\bdomega}|_{g_{\bdomega}}\\
&\le |P^{-1}\bdomega-\bdomega|_{g_{\bdomega}}+|\bdomega-\tilde{\bdomega}|_{g_{\bdomega}}\\
&= |(I_3-P)\bdomega'|_{g_{\bdomega'}}+|\bdomega-\tilde{\bdomega}|_{g_{\bdomega}}
= O(\delta_2+\delta_1).
\end{align*}
Next, we give the upper bound for $|g_{\bdomega'}-g_{\tilde{\bdomega}}|_{g_{\bdomega}}$. 
For $u,v\in V$, Lemma \ref{lem norm asso} yields 
\begin{align*}
\frac{1}{2}|S_{\bdomega'}(u,v)(\omega'_1)^2-S_{\tilde{\bdomega}}(u,v)\tilde{\omega}_1^2|_{g_{\bdomega}}
&=\left| \iota_u\omega'_1\wedge\iota_v\omega'_2\wedge\omega'_3
- \iota_u\tilde{\omega}_1\wedge\iota_v\tilde{\omega}_2\wedge\tilde{\omega}_3\right|_{g_{\bdomega}}\\
&\le C^2\left| \iota_u\omega'_1-\iota_u\tilde{\omega}_1\right|_{g_{\bdomega}}|\iota_v\omega'_2|_{g_{\bdomega}}|\omega'_3|_{g_{\bdomega}}\\
&\quad\quad 
+C^2|\iota_u\tilde{\omega}_1|_{g_{\bdomega}}\left| \iota_u\omega'_2-\iota_u\tilde{\omega}_2\right|_{g_{\bdomega}}|\omega'_3|_{g_{\bdomega}}\\
&\quad\quad 
+ C^2|\iota_u\tilde{\omega}_1|_{g_{\bdomega}}|\iota_v\tilde{\omega}_2|_{g_{\bdomega}}|\omega'_3-\tilde{\omega}_3|_{g_{\bdomega}}.
\end{align*}
Then by Lemma \ref{lem norm asso}, 
\begin{align*}
\frac{1}{2}|S_{\bdomega'}(u,v)(\omega'_1)^2-S_{\tilde{\bdomega}}(u,v)\tilde{\omega}_1^2|_{g_{\bdomega}}
&\le O(\delta_1+\delta_2)|u|_{g_{\bdomega}}|v|_{g_{\bdomega}}.
\end{align*}
Similarly, we also have $|(\omega'_1)^2-\tilde{\omega}_1^2|_{g_{\bdomega}}= O(\delta_1+\delta_2)$, which gives 
\begin{align*}
|S_{\bdomega'}(u,v)-S_{\tilde{\bdomega}}(u,v)|
&= O(\delta_1+\delta_2)|u|_{g_{\bdomega}}|v|_{g_{\bdomega}}.
\end{align*}
If $\delta_1,\delta_2$ are sufficiently small, 
then the symmetric tensor $S_{\bdomega'}$ is close to $S_{\tilde{\bdomega}}$, hence $S_{\bdomega'}$ is positive (resp.~negative) if 
$S_{\tilde{\bdomega}}$ is positive (resp.~negative). 
Consequently, we have 
\begin{align}
|g_{\bdomega}(u,v)-g_{\tilde{\bdomega}}(u,v)|
&= O(\delta_1+\delta_2)|u|_{g_{\bdomega}}|v|_{g_{\bdomega}},\label{ineq g omega' - g tilde omega}
\end{align}
which gives  
$(1-C(\delta_1+\delta_2))g_{\bdomega}\le g_{\tilde{\bdomega}} \le (1+C(\delta_1+\delta_2))g_{\bdomega}$ for a sufficiently large positive constant $C$. 

If $|\bdomega-\tilde{\bdomega}|_{g_{\tilde{\bdomega}}}\le \delta_1$ for sufficiently small $\delta_1$, then $|Q_{\bdomega}-I_3|=O(\delta_1)$. We define $\bdomega'$ and $P$ in the same way, then we obtain 
\begin{align*}
|\bdomega'-\tilde{\bdomega}|_{g_{\tilde{\bdomega}}}
&=|P^{-1}\bdomega-P^{-1}\tilde{\bdomega}|_{g_{\tilde{\bdomega}}}
+|(P^{-1}-I_3)\tilde{\bdomega}|_{g_{\tilde{\bdomega}}}\\
&\le (1+O(\delta_1))|\bdomega-\tilde{\bdomega}|_{g_{\tilde{\bdomega}}}
-|P^{-1}-I_3|O(1)=O(\delta_1).
\end{align*}
The rest of the argument is similar. 
\end{proof}

\section{Hyper-K\"ahler metrics on the $K3$ surface collapsing to $3$-dimensional spaces}\label{sec Foscolo}
We review a family of 
hyper-K\"ahler metrics on the $K3$ surface 
collapsing to $3$-dimensional spaces 
constructed by Foscolo in \cite{Foscolo2019}. 

\subsection{Gibbons-Hawking ansatz}\label{subsec GH}
First of all, we review the construction of 
hyper-K\"ahler $4$ manifolds equipped with 
tri-hamiltonian $U(1)$-actions introduced by \cite{GH1978}. 

On the $3$-dimensional torus $\T:=\R^3/\Z^3$, 
there are closed $1$-forms $\theta_1,\theta_2,\theta_3\in\Omega^1(\T)$ which form a frame of the cotangent bundle $T^*\T$. 
We put 
\begin{align*}
\boldsymbol{\theta}=(\theta_1,\theta_2,\theta_3)
\end{align*}
and define a flat Riemannian metric by $g_{\boldsymbol{\theta}}:=\theta_1^2+\theta_2^2+\theta_3^2$.  
Let $\hat{Y}\subset \T$ be an open subset and 
$h\in C^\infty(\hat{Y})$ a smooth positive valued harmonic function on $\hat{Y}$, then we can see that $*dh$ is a closed 
$2$-form on $\hat{Y}$, where $*$ is the Hodge star operator with respect to $g_{\boldsymbol{\theta}}$. 
Here, we assume that the cohomology class 
represented by $*dh$ is in $H^2(\hat{Y},\Z)$. 
Then there is a principal $U(1)$-bundle $\hat{\mu}_h\colon 
\hat{M}_h\to \hat{Y}$ and a $U(1)$-connection $\sqrt{-1}\theta\in \Omega^1(\hat{M}_h,\sqrt{-1}\R)$ whose curvature form is 
$\sqrt{-1}\hat{\mu}_h^*(*dh)$, that is, we have 
$d\theta=\hat{\mu}_h^*(*dh)$. 
Here, the pullback of $h$ and $\theta_i$ 
by $\hat{\mu}_h$ can be regarded as 
functions or $1$-forms on $\hat{M}_h$, using the same notation for brevity. 
Then we have closed $2$-forms 
\begin{align*}
\hat{\eta}_{h,1}&:=\theta_1\wedge \theta+h\cdot \theta_2\wedge \theta_3,\\
\hat{\eta}_{h,2}&:=\theta_2\wedge\theta+h\cdot \theta_3\wedge \theta_1,\\
\hat{\eta}_{h,3}&:=\theta_3\wedge \theta+h\cdot \theta_1\wedge \theta_2
\end{align*}
with relations 
\begin{align*}
(\hat{\eta}_{h,1})^2&=(\hat{\eta}_{h,2})^2=(\hat{\eta}_{h,3})^2,\\
\hat{\eta}_{h,i}\wedge\hat{\eta}_{h,j}&=0\quad (i\neq j)
\end{align*}
on $M_h$. 
Then 
\begin{align*}
\hat{\boldsymbol{\eta}}_h:=\left( \hat{\eta}_{h,1},\,
\hat{\eta}_{h,2},\,
\hat{\eta}_{h,3}\right)
\end{align*}
is a hyper-K\"ahler triple on $\hat{M}_h$ 
and induces a hyper-K\"ahler metric 
\begin{align*}
g_{\hat{\boldsymbol{\eta}}_h}=h^{-1}\theta^2+h\sum_{i=1}^3\theta_i^2.
\end{align*}

\subsection{Hyper-K\"ahler metrics on the $K3$ surface}\label{subsec Foscolo}
There is a natural $\{ \pm 1\}$-action on $\T$ 
defined by $\pm[x]:=[\pm x]$, where 
$[x]\in\T$ is the equivalence class represented by $x\in\R^3$. 
Then the fixed point set is given by 
\begin{align*}
\T^{\{ \pm 1\}}&=\left\{ [a_1,a_2,a_3]\in\T\left|\, a_i=0\mbox{ or }\frac{1}{2}
\right.\right\}
=:\{ q_1,\ldots,q_8\}.
\end{align*}
Let $n$ be a nonnegative integer. 
We take finitely many mutually distinct $2n$ 
points 
$p_1,\ldots,p_n,-p_1,\ldots,-p_n\in \T\setminus\T^{\{ \pm 1\}}$ and 
put 
\begin{align*}
S&:=\{ q_1,\ldots,q_8,p_1,\ldots,p_n,-p_1,\ldots,-p_n\},\\
\T^*&:=\T\setminus S.
\end{align*}
For 
\begin{align*}
\bm{m}&:=(m_1,\ldots,m_8)\in (\Z_{\ge 0})^8,\\
\bm{k}&:=(k_1,\ldots,k_n)\in (\Z_{>0})^n,
\end{align*}
we say that 
$(\bm{m},\bm{k})=(m_1,\ldots,m_8,k_1,\ldots,k_n)$ 
satisfy the {\it balancing condition} if 
\begin{align}
\sum_{j=1}^8m_j+\sum_{i=1}^nk_i=16.\label{eq balance}
\end{align}

Let $d_{g_{\boldsymbol{\theta}}}$ be the Riemannian distance of $g_{\boldsymbol{\theta}}$. 
For $r>0$, put 
\begin{align*}
B_\T(p,r):=\{ x\in\T|\, d_{g_{\boldsymbol{\theta}}}(p,x)<r\}.
\end{align*}
If we put 
\begin{align*}
\rho_p(x):= d_{g_{\boldsymbol{\theta}}}(p,x), 
\end{align*}
then $1/\rho_p$ is a harmonic function for $g_{\boldsymbol{\theta}}$ on $B_\T(p,r)\setminus\{ p\}$, where $r>0$ is a sufficiently small positive number. 
For an open set $U\subset \T$, a function $f\colon U\to \R$, $p\in\T$ and $k\in\R$, we write 
\begin{align*}
f\sim_p \frac{k}{\rho_p},
\end{align*}
if there are a sufficiently small $\delta>0$ and a smooth function 
$\varphi\colon B_\T(p,\delta)\to \R$ such that 
$B_\T(p,\delta)\setminus \{ p\}\subset U$ and 
\begin{align*}
f = \frac{k}{\rho_p}+\varphi
\end{align*}
on $B_\T(p,\delta)\setminus \{ p\}$.

\begin{thm}[{\cite[Proposition 4.3]{Foscolo2019}}]
Assume that the balancing condition \eqref{eq balance} 
holds. Then there are a $\{ \pm 1\}$-invariant harmonic function $h\colon \T^*\to \R$ and a $U(1)$-connection 
$\sqrt{-1}\theta$ on $\hat{M}_h\to \hat{Y}=\T^*$ 
such that 
\begin{align*}
h&\sim_{p_i} \frac{k_i}{2\rho_{p_i}},\quad
h\sim_{-p_i} \frac{k_i}{2\rho_{-p_i}},\quad
h\sim_{q_j} \frac{2m_j-4}{2\rho_{q_j}},
\end{align*}
for all $i=1,\ldots,n$, $j=1,\ldots,8$ and 
\begin{align*}
[*dh]&\in H^2(\T^*,\Z)\\
d\theta&=\hat{\mu}_h^*(*dh).
\end{align*}
Moreover, there is 
a smooth map $\tau\colon \hat{M}_h\to M_h$ such that 
$\hat{\mu}_h\circ\tau=-\hat{\mu}_h$, $\tau^2=1$ and 
$\tau^*\theta=-\theta$. 
\label{thm harm pole}
\end{thm}
\begin{rem}
\normalfont
For $S\subset \T$ and $(\bm{m},\bm{k})\in\Z^{n+8}$ 
with the balancing condition \eqref{eq balance}, 
we denote by $\F(S,\bm{m},\bm{k})$ the set consisting of 
$(h,\theta)$ satisfying the conclusion of 
Theorem \ref{thm harm pole}. 
Here, $(h,\theta)$ is not determined uniquely 
from $(S,\bm{m},\bm{k})$ since 
$h$ has the ambiguity of additive constants and $\theta$, 
up to the gauge equivalence, has 
the ambiguity of parallel $1$-forms on $\T$. 
\end{rem}

By the above Theorem, 
$\hat{\boldsymbol{\eta}}_h$ is $\tau$-invariant, 
accordingly, it induces a triple of $2$-forms 
$\boldsymbol{\eta}_h$ on 
the quotient manifold $M_h:=\hat{M}_h/\tau$. 
Moreover, $\hat{\mu}_h$ induces a map 
\begin{align*}
\mu_h\colon M_h\to \T^*/\{ \pm 1\}.
\end{align*}
Here, $\boldsymbol{\eta}_h$ defines a hyper-K\"ahler triple on the region where $h$ is positive. 

Denote by $\bar{p}\in\T/\{ \pm 1\}$ the image of 
$p\in\T$ by the quotient map $\T\to \T/\{ \pm 1\}$. 
We denote by $B_\T(\bar{p},r)\subset \T/\{ \pm 1\}$ the image of 
$B_\T(p,r)$. 
Put 
\begin{align*}
\delta_0:=\frac{1}{4}\inf\{ d_{g_{\boldsymbol{\theta}}}(x,y),{\rm inj}(g_{\boldsymbol{\theta}})|\, x,y\in S,\, x\neq y\}>0.
\end{align*}
For $p\in S$ and $r<2\delta_0$, 
the restriction of $\mu_h$ to $\mu_h^{-1}(B(p,r))$ is given as follows.

Let $\nu_{S^3}\colon S^3\to S^2$ be the Hopf fibration. 
Here, we regard $S^3$ as the unit sphere in $\C^2$ centered at the origin, then $\nu_{S^3}$ is given by 
\begin{align*}
\nu_{S^3}(z,w):=\left( |z|^2-|w|^2,\, 2{\rm Re}(zw),\, 2{\rm Im}(zw)\right).
\end{align*}
Then $\nu_{S^3}$ is the principal $U(1)$-bundle 
by the action 
\begin{align*}
(z,w)\cdot\lambda:=(z\lambda,w\lambda^{-1})
\end{align*}
for $(z,w)\in S^3$ and $\lambda\in U(1)$. 
Let $\Z_k\subset U(1)$ be the cyclic subgroup of order $k$. 
Then the map $\nu_{S^3}$ induces a $U(1)$-fiber bundle  $S^3/\Z_k\to S^2$. Here, we can define a $U(1)$-action preserving every fiber of $S^3/\Z_k\to S^2$ by $[z,w]_k\cdot\lambda:=[z\lambda^{1/k},w\lambda^{-1/k}]_k$ for $[z,w]_k\in S^3/\Z_k$ and $\lambda\in U(1)$, 
where $[z,w]_k$ is the equivalence class represented by $(z,w)$.

For an integer $k$, 
put $W^k:=S^3/\Z_k$ and put 
\begin{align*}
H^k:=W^k\times\R_+.
\end{align*}
Here, $W^{-k}=W^k$ as manifolds but the orientation of the  $U(1)$-action on $W^{-k}$ is reversed comparing with that of $W^k$. 
Then $\nu_{S^3}$ induces the map 
$\nu_{A_{k-1}}\colon H^k\to S^2\times \R_+$ defined by 
\begin{align*}
\nu_{A_{k-1}}([z,w]_k,r):= \left( \nu_{S^3}(z,w),\, r\right).
\end{align*}
If $k$ is even, then we have a $\{ \pm 1\}$-action 
on $H^k$ defined by 
\begin{align*}
(-1)\cdot ([z,w]_k,r):=([w,-z]_k,r).
\end{align*}
Then we obtain a map $\nu_{D_m} \colon H^{2(m-2)}/\{ \pm 1\}\to S^2/\{ \pm 1\}\times\R_+$ defined by 
\begin{align*}
\nu_{D_m}((\pm 1)( [z,w]_{2(m-2)},r)):= \left( \pm\nu_{S^3}(z,w),\, r\right) 
\in S^2/\{ \pm 1\}\times\R_+.
\end{align*}

For $0\le r_1<r_2<\delta_0$, put 
\begin{align*}
A_\T(p,r_1,r_2)
&:=\{ x\in\T|\, r_1<d_{g_{\boldsymbol{\theta}}}(p,x)<r_2\},
\end{align*}
and denote by $A_\T(\bar{p},r_1,r_2)$ the image of 
$A_\T(p,r_1,r_2)$ by the projection 
$\T\to \T/\{ \pm 1\}$. 
For an open interval $(a,b)\subset \R_+$, we put 
\begin{align*}
H^k_{(a,b)}:=W^k\times (a,b)\subset H^k.
\end{align*}
Note that $S^2\times \R_+$ can be isometrically 
embedded in $\R^3\setminus\{ 0\}$ by $(u,r)\mapsto ru$. 
If $h$ is given by Theorem \ref{thm harm pole}, 
then we have the identifications 
as $C^\infty$ manifolds
\begin{align}
\mu_h^{-1}\left( A_\T(\bar{p}_i,0,R)\right)\cong H^{k_i}_{(0,R)},\quad
\mu_h^{-1}\left( A_\T(\bar{q}_j,0,R)\right)\cong H^{2(m_j-2)}_{(0,R)}/\{ \pm 1\}\label{diff glue},
\end{align}
and isometries
\begin{align*}
A_\T(\bar{p}_i,0,R)\cong S^2\times (0,R),\quad
A_\T(\bar{q}_j,0,R)\cong S^2/\{ \pm 1\}\times (0,R).
\end{align*}
Here, the origin of $\R^3$ and $\R^3/\{ \pm 1\}$ are identified with $\bar{p}_i$ are $\bar{q}_j$, respectively. 
Moreover, under the identifications \eqref{diff glue}, 
we have the following identifications 
\begin{align}
\mu_h|_{\mu_h^{-1}\left( A_\T(\bar{p}_i,0,R)\right)}
&\cong \nu_{A_{k_i-1}}|_{H^{k_i}_{(0,R)}},\quad
\mu_h|_{\mu_h^{-1}\left( A_\T(\bar{q}_j,0,R)\right)}
\cong \nu_{D_{m_j}}|_{H^{2(m_j-2)}_{(0,R)}/\{ \pm 1\}},\label{diff glue moment}
\end{align}
as principal $U(1)$-bundles. 

Next, we obtain a smooth and compact $4$-manifold 
by gluing $M_h$ and the underlying manifolds of 
ALF gravitational instantons. 
An {\it ALF gravitational instanton} is 
a $4$-dimensional complete Riemannian manifold 
$(M,g)$ satisfying the following properties. 
\begin{itemize}
\setlength{\parskip}{0cm}
\setlength{\itemsep}{0cm}
 \item[$({\rm i})$] There is a compact set $K\subset M$, 
 positive constant $R>0$ 
and a diffeomorphism 
$\phi\colon H^k_{(R,\infty)}\to M\setminus K$ 
or $\phi\colon H^{2(m-2)}_{(R,\infty)}/\{ \pm 1\}\to M\setminus K$. 
The former case is said to be {\it of type $A_{k-1}$}, 
the latter case is said to be {\it of type $D_m$}. 
 \item[$({\rm ii})$] 
 Let $g_k$ be the hyper-K\"ahler metric on $H^k_{(R,\infty)}$ (or $H^k_{(R,\infty)}/\{\pm 1\}$) given by the harmonic function $x\mapsto 1+k/2|x|$ via Gibbons-Hawking ansatz and $r$ be the composition of the projection 
$H^k_{(R,\infty)}\to (R,\infty)$ (or $H^k_{(R,\infty)}/\{\pm 1\}\to (R,\infty)$) and $\phi^{-1}$. 
Then $\phi^*g$ satisfies the following  
 asymptotic condition, 
\begin{align*}
|\nabla_{g_k}^l(\phi^*g-g_k)|_{g_k}&=O(r^{-3-l})\quad (\mbox{type }A_k),\\
|\nabla_{g_{2m-2}}^l(\phi^*g-g_{2m-2})|_{g_{2m-2}}&=O(r^{-3-l})\quad (\mbox{type }D_m),
\end{align*}
for each $l\ge 0$.  
 \item[$({\rm iii})$] There is a hyper-K\"ahler triple 
$\boldsymbol{\eta}=(\eta_1,\eta_2,\eta_3)\in\Omega^2(M)\otimes\R^3$ whose hyper-K\"ahler metric is $g$.
\end{itemize}
See \cite[Definition 3.6]{Foscolo2019} for the 
precise statement of the asymptotic condition.

Moreover, we have a smooth map 
\begin{align*}
\nu&\colon M\setminus K\to S^2\times (R,\infty),\\
(\mbox{resp. }\nu&\colon M\setminus K\to S^2/\{ \pm 1\}\times (R,\infty))
\end{align*}
such that 
\begin{align}
\nu\circ \phi=\nu_{A_{k-1}}\quad
(\mbox{resp. } \nu\circ \phi=\nu_{D_m}).\label{diff ALF moment}
\end{align}
By the asymptotic condition $({\rm ii})$, 
the limit of the length of circle fibers of $\nu$ at infinity is 
normalized to be $2\pi$. 
When we refer to ALF gravitational instantons, we often 
refer to their additional structures, such as $(M,g,\boldsymbol{\eta},K,R,\phi,\nu)$.  
For $R'\ge R$, we put 
\begin{align*}
M|_{<R'}&:=\nu^{-1}(S^2\times (R,R'))\cup K,\\
(\mbox{resp. }M|_{<R'}&:=\nu^{-1}(S^2/\{ \pm 1\}\times (R,R'))\cup K).
\end{align*}
Then $M|_{<R'}$ is an open submanifold of $M$.

Put 
\begin{align*}
\hat{M}_h|_{>R}&:=\left\{ \left. 
x\in \hat{M}_h\right| \, \min_{i,j}\left\{ 
d_\T(\hat{\mu}_h(x),p_i),\, d_\T(\hat{\mu}_h(x),q_j)\right\}>R\right\},\\
M_h|_{>R}&:=\hat{M}_h|_{>R}/\tau.
\end{align*}
We can also define $M|_{\le R'},M_h|_{\ge R}$ in the same way. 

\begin{thm}[{\cite{Foscolo2019}}]
Let $(\bm{m},\bm{k})\in \Z^{n+8}$ satisfy the balancing condition \eqref{eq balance} and $(h,\theta)\in\F(S,\bm{m},\bm{k})$. 
There is a constant $\varepsilon_0>0$ and ALF gravitational instantons 
\begin{align*}
&(M^A_i,g^A_i,\boldsymbol{\eta}^A_i,K^A_i,R_0,\phi^A_i,\mu^A_i)\quad (\mbox{type }A_{k_i-1},\, i=1,\ldots,n),\\
&(M^D_j,g^D_j,\boldsymbol{\eta}^D_j,K^D_j,R_0,\phi^D_j,\mu^D_j)
\quad(\mbox{type }D_{m_j},\, j=1,\ldots,8),
\end{align*}
depending only on 
$n,(\bm{m},\bm{k}), (h,\theta)$ such that the following hold. 
For any $0<\varepsilon\le \varepsilon_0$, 
there is a smooth $4$-manifold $X$, definite triples of closed 
$2$-forms $\boldsymbol{\omega}_\varepsilon
=\left( 
\omega_{\varepsilon,1},\, \omega_{\varepsilon,2},\, \omega_{\varepsilon,3}\right)$, 
hyper-K\"ahler triples $\widetilde{\boldsymbol{\omega}}_\varepsilon
=\left( 
\widetilde{\omega}_{\varepsilon,1},\, \widetilde{\omega}_{\varepsilon,2},\, \widetilde{\omega}_{\varepsilon,3}\right)$
on $X$ 
and open covers $\{ U_\varepsilon, U^A_{i,\varepsilon}, U^D_{j,\varepsilon}|\, i=1,\ldots,n,\, j=1,\ldots,8\}$ of $X$ such that 
\begin{itemize}
\setlength{\parskip}{0cm}
\setlength{\itemsep}{0cm}
 \item[$({\rm i})$] 
$U_\varepsilon=M_{\varepsilon^{-1}+h}|_{>\varepsilon^{2/5}}$, 
$U^A_{i,\varepsilon}=M^A_i|_{<2\varepsilon^{-3/5}}$ and $U^D_{j,\varepsilon}=M^D_j|_{<2(1+\lambda_j\varepsilon)\varepsilon^{-3/5}}$ for some constant $\lambda_j$ determined by $h$, 
 \item[$({\rm ii})$] $\boldsymbol{\omega}_\varepsilon
 = \varepsilon\boldsymbol{\eta}_{\varepsilon^{-1}+h}$ on 
$M_{\varepsilon^{-1}+h}|_{\ge 2\varepsilon^{2/5}}$ and $\varepsilon^{-1}+h\ge \varepsilon^{-1}/2$ on $U_{\varepsilon}$, 
 \item[$({\rm iii})$] there are diffeomorphisms 
$U_\varepsilon\cap U^A_{i,\varepsilon}\cong H^{k_i}_{(\varepsilon^{2/5},2\varepsilon^{2/5})}$ which identifies $\mu_h-\bar{p}_i$ with $\varepsilon \mu^A_i$, 
$U_\varepsilon\cap U^D_{j,\varepsilon}\cong H^{2(m_j-2)}_{(\varepsilon^{2/5},2\varepsilon^{2/5})}/\{ \pm 1\}$ which identifies $\mu_h-\bar{q}_j$ with $\varepsilon(1+\lambda_j\varepsilon)^{-1} \mu^D_j$, 
and we have $U^A_{i,\varepsilon}\cap U^D_{j,\varepsilon}=\emptyset$ for all $i,j$, 
 \item[$({\rm iv})$] for all $i$, 
 we have $\boldsymbol{\omega}_\varepsilon
 = \varepsilon^2(\boldsymbol{\eta}^A_i+\varepsilon\boldsymbol{\eta}'_i+\varepsilon^2\boldsymbol{\eta}''_i)$ on 
$M^A_i|_{\le \varepsilon^{-3/5}}$, 
for some $\boldsymbol{\eta}'_i,\boldsymbol{\eta}''_i\in\Omega^2(M^A_i)\otimes\R^3$, which are independent of 
$\varepsilon$, 
 \item[$({\rm v})$] for all $j$, 
 we have $\boldsymbol{\omega}_\varepsilon
 = \varepsilon^2(1+\varepsilon\lambda_j)^{-1}\boldsymbol{\eta}^D_j$ 
 on $M^D_j|_{\le (1+\lambda_j\varepsilon)\varepsilon^{-3/5}}$, 
 \item[$({\rm vi})$] there is a positive constant $C>0$ 
such that 
$\varepsilon^2(1-C\varepsilon)
g^A_i\le g_{\boldsymbol{\omega}_\varepsilon}\le \varepsilon^2(1+C\varepsilon)
g^A_i$ and 
$\varepsilon^2(1-C\varepsilon)
g^D_j\le g_{\boldsymbol{\omega}_\varepsilon}\le \varepsilon^2(1+C\varepsilon)
g^D_j$ on $U_\varepsilon\cap U^A_{i,\varepsilon}$ 
and $U_\varepsilon\cap U^D_{j,\varepsilon}$, respectively, 
and 
$(1-C\varepsilon^{11/10})
g_{\widetilde{\boldsymbol{\omega}}_\varepsilon}\le g_{\boldsymbol{\omega}_\varepsilon}\le (1+C\varepsilon^{11/10})
g_{\widetilde{\boldsymbol{\omega}}_\varepsilon}$ on $X$.
\end{itemize}
\label{thm k3 to t3}
\end{thm}
\begin{proof}
We can prove directly from results in \cite{Foscolo2019}. 
Here, we give the outline for the reader's convenience. 
First of all, by \cite[Lemma 4.9]{Foscolo2019}, 
we have $\varepsilon^{-1}+h\ge \varepsilon^{-1}/2$ 
on the complement of $8\varepsilon$-balls of $\{ q_1,\ldots,q_8\}$ 
for sufficiently small $\varepsilon$. 
Therefore, if we take 
sufficiently small $\varepsilon_0>0$, 
we can see that $\varepsilon^{-1}+h\ge \varepsilon^{-1}/2$ on 
$\T\setminus\bigcup_{j=1}^8B_\T(q_j,\varepsilon^{2/5})$ for all 
$0<\varepsilon\le \varepsilon_0$. 
We can glue 
$M_{\varepsilon^{-1}+h}|_{>\varepsilon^{2/5}}$ and $M^A_i|_{<2\varepsilon^{-3/5}}$ as follows for all $0<\varepsilon\le \varepsilon_0$. 
By taking $\varepsilon_0$ sufficiently small, we may suppose $R_0\le \varepsilon^{-3/5}$. 
Accordingly, we can take open subsets 
\begin{align*}
\mu_h^{-1}(A_\T(\bar{p}_i,\varepsilon^{2/5},2\varepsilon^{2/5}))&\subset M_{\varepsilon^{-1}+h}|_{>\varepsilon^{2/5}},\\
(\mu^A_i)^{-1}(S^2\times (\varepsilon^{-3/5},2\varepsilon^{-3/5}))
&\subset M^A_i|_{<2\varepsilon^{-3/5}}.
\end{align*}
Now, we have an isomorphism 
\begin{align*}
A_\T(\bar{p}_i,\varepsilon^{2/5},2\varepsilon^{2/5}))
\stackrel{\cong}{\to}S^2\times (\varepsilon^{-3/5},2\varepsilon^{-3/5}),\quad 
v\mapsto \varepsilon^{-1}(v-p_i),
\end{align*}
which lifts to a diffeomorphism 
\begin{align*}
\mu_h^{-1}(A_\T(\bar{p}_i,\varepsilon^{2/5},2\varepsilon^{2/5}))
\cong
(\mu^A_i)^{-1}(S^2\times (\varepsilon^{-3/5},2\varepsilon^{-3/5}))
\end{align*}
by considering the composition of $\phi^A_i$ and \eqref{diff glue}. 
Then, under this identification, \eqref{diff glue moment} and \eqref{diff ALF moment} gives 
the identification $\mu_h-\bar{p}_i\cong\varepsilon \mu^A_i$. 
Similarly, we have the identification 
\begin{align*}
\mu_h^{-1}(A_\T(\bar{q}_j,\varepsilon^{2/5},2\varepsilon^{2/5}))
&\cong
(\mu^D_j)^{-1}(S^2\times ((1+\varepsilon\lambda_j)\varepsilon^{-3/5},2(1+\varepsilon\lambda_j)\varepsilon^{-3/5})),\\
\mu_h-\bar{q}_j&\cong\varepsilon (1+\varepsilon\lambda_j)^{-1}\mu^D_j. 
\end{align*}
Here, $\lambda_j\in\R$ is the constant term of the Taylor expansion of $h-(2m_j-4)/2\rho_{q_j}$ at $q_j$. 
We will later mention why the term $(1+\varepsilon\lambda_j)$ is necessary. 
By taking $\varepsilon$ smaller than $\delta_0^{5/2}$, 
we obtain a smooth compact $4$-manifold $X$ 
with an open cover $\{ U_\varepsilon, U^A_{i,\varepsilon}, U^D_{j,\varepsilon}|\, i=1,\ldots,n,\, j=1,\ldots,8\}$ 
such that $({\rm i,iii})$ hold.

The construction of $\boldsymbol{\omega}_\varepsilon$ 
is given by \cite[Section 5.2]{Foscolo2019}, then we have 
$({\rm ii,v})$. 
On $M^A_i$, the hyper-K\"ahler triple $\boldsymbol{\eta}^A_i$ is given by the Gibbons-Hawking ansatz, which is obtained from a positive valued harmonic function on $\R^3$ with $k_i$ singularities. 
To obtain $\boldsymbol{\omega}_\varepsilon$ on 
$M^A_i|_{\le \varepsilon^{-3/5}}$, we add this harmonic function 
and $\varepsilon \lambda_i+\varepsilon^2l_i$, where $\lambda_i$ is a constant and $l_i$ is a linear function 
such that $\lambda_i+l_i$ is the 1st order approximation of 
the Taylor expansion of $h-k/2\rho_{p_i}$ at $p_i$. 
Then the new function is positive on $M^A_i|_{<2\varepsilon^{-3/5}}$ if $\varepsilon_0$ is sufficiently small, 
we obtain the new 
hyper-K\"ahler triple by Gibbons-Hawking ansatz, 
which we may write $\boldsymbol{\eta}^A_i+\varepsilon\boldsymbol{\eta}'_i+\varepsilon^2\boldsymbol{\eta}''_i$ for some triples of 
$2$-forms $\boldsymbol{\eta}'_i,\boldsymbol{\eta}''_i$ independent of $\varepsilon$. 
This implies $({\rm iv})$. 
$\boldsymbol{\omega}_\varepsilon$ on $M^D_j|_{\le (1+\lambda_j\varepsilon)\varepsilon^{-3/5}}$ is defined by 
$({\rm v})$.

The description of $\boldsymbol{\omega}_\varepsilon$ 
on the intersections 
$U_\varepsilon\cap U^A_{i,\varepsilon}$ 
and $U_\varepsilon\cap U^D_{j,\varepsilon}$ 
are given by \cite[(5.7)]{Foscolo2019}. 
There are triples of $1$-forms $\bm{b}_{i,\varepsilon},\bm{b}_{j,\varepsilon}$ on 
$U_\varepsilon\cap U^A_{i,\varepsilon}$ 
and $U_\varepsilon\cap U^D_{j,\varepsilon}$, respectively, 
we may write 
\[ \varepsilon\boldsymbol{\eta}_{\varepsilon^{-1}+h}=
\left \{
\begin{array}{cc}
\varepsilon^2(\boldsymbol{\eta}^A_i+\varepsilon\boldsymbol{\eta}'_i+\varepsilon^2\boldsymbol{\eta}''_i)+d\bm{b}_{i,\varepsilon} & \mbox{ on }U_\varepsilon\cap U^A_{i,\varepsilon}, \\
\varepsilon^2(1+\varepsilon\lambda_j)^{-1}\boldsymbol{\eta}^D_j+d\bm{b}_{j,\varepsilon} & \mbox{ on }U_\varepsilon\cap U^D_{j,\varepsilon},
\end{array}
\right.
\]
and we have 
\[ \boldsymbol{\omega}_\varepsilon=
\left \{
\begin{array}{cc}
\varepsilon^2(\boldsymbol{\eta}^A_i+\varepsilon\boldsymbol{\eta}'_i+\varepsilon^2\boldsymbol{\eta}''_i)+d(\chi_{p_i}\bm{b}_{i,\varepsilon}) & \mbox{ on }U_\varepsilon\cap U^A_{i,\varepsilon} \\
\varepsilon^2(1+\varepsilon\lambda_j)^{-1}\boldsymbol{\eta}^D_j+d(\chi_{q_j}\bm{b}_{j,\varepsilon}) & \mbox{ on }U_\varepsilon\cap U^D_{j,\varepsilon}
\end{array}
\right.
\]
Here, $\chi_{p_i}$ is a smooth cut-off function on 
$U_\varepsilon\cap U^A_{i,\varepsilon}=W^{k_i}\times (\varepsilon^{2/5},2\varepsilon^{2/5})$ 
such that $\chi_{p_i}(w,\varepsilon^{2/5})=0$, 
$\chi_{p_i}(w,2\varepsilon^{2/5})=1$, 
$0\le \chi_{p_i}\le 1$ and 
$|\nabla \chi_{p_i}|\le C\rho_{p_i}^{-1}$ 
for a positive constant $C$. 
From now on, we will replace $C$ with the larger one 
if it is necessary. 
Similarly, $\chi_{q_j}$ is a smooth cut-off function on 
$U_\varepsilon\cap U^D_{j,\varepsilon}=W^{2(m_j-2)}/\{ \pm 1\}\times (\varepsilon^{2/5},2\varepsilon^{2/5})$ 
such that $\chi_{q_j}(w,\varepsilon^{2/5})=0$, 
$\chi_{q_j}(w,2\varepsilon^{2/5})=1$, 
$0\le \chi_{q_j}\le 1$ and 
$|\nabla \chi_{q_j}|\le C\rho_{q_j}^{-1}$. 
We also have 
\begin{align*}
|\nabla^l\bm{b}_{i,\varepsilon}|&\le C\max
\left\{ \varepsilon^3\rho_{p_i}^{-2-l},\, \varepsilon\rho_{p_i}^{3-l}\right\},\\
|\nabla^l\bm{b}_{j,\varepsilon}|&\le C\max
\left\{ \varepsilon^3\rho_{q_j}^{-2-l},\, \varepsilon\rho_{q_j}^{3-l}\right\},
\end{align*}
for $l=0,1$ by \cite[Lemma 5.4 and (5.5b)]{Foscolo2019}. 
The above estimates are necessary to show the existence of the hyper-K\"ahler triple $\tilde{\boldsymbol{\omega}}_\varepsilon$. 
This is why $\lambda_i,l_i,\lambda_j$ are considered. 
If we ignore these terms, then the above estimates get worse. 
Since $\varepsilon^{2/5}\le \rho_{p_i},\rho_{q_j}\le 2\varepsilon^{2/5}$, 
we have 
\begin{align*}
|\nabla^l\bm{b}_{i,\varepsilon}|\le C\varepsilon^{11/5-2l/5},\quad
|\nabla^l\bm{b}_{j,\varepsilon}|\le C \varepsilon^{11/5-2l/5}.
\end{align*}
Therefore, we have 
\begin{align*}
|d(\chi_{p_i}\bm{b}_{i,\varepsilon})|
\le C\varepsilon^{9/5},\quad
|d(\chi_{q_j}\bm{b}_{j,\varepsilon})|
\le C\varepsilon^{9/5}.
\end{align*}
In the above inequalities, the norms of $1$-forms are 
given by $\varepsilon^2g^A_i$ and $\varepsilon^2g^D_j$, 
respectively. 
Therefore, we have 
\begin{align*}
\left| \boldsymbol{\omega}_\varepsilon - \varepsilon^2\boldsymbol{\eta}^A_i\right|
&\le \left| \varepsilon^3\boldsymbol{\eta}'_i+\varepsilon^4\boldsymbol{\eta}''_i\right|+C\varepsilon^{9/5}
= O(\varepsilon+\varepsilon^2+\varepsilon^{9/5})=O(\varepsilon),\\
\left| \boldsymbol{\omega}_\varepsilon - \varepsilon^2\boldsymbol{\eta}^D_j\right|
&\le \left| \boldsymbol{\omega}_\varepsilon - \varepsilon^2(1+\varepsilon\lambda_j)^{-1}\boldsymbol{\eta}^D_j\right|+\left|\varepsilon^2\{ (1+\varepsilon\lambda_j)^{-1} -1\}\boldsymbol{\eta}^D_j\right|\\
&=O(\varepsilon^{9/5})+O(\varepsilon)=O(\varepsilon).
\end{align*}
Then by Lemma \ref{lem norm met}, we have the first and the second 
inequality of $({\rm vi})$. 
Finally, we estimate the difference between $\widetilde{\boldsymbol{\omega}}_\varepsilon$ and 
$\boldsymbol{\omega}_\varepsilon$. 
By the argument in the proof of \cite[Theorem 6.17]{Foscolo2019}, we may write 
$\widetilde{\boldsymbol{\omega}}_\varepsilon
=\boldsymbol{\omega}_\varepsilon+d\bm{a}_\varepsilon+\boldsymbol{\zeta}_\varepsilon$ 
for some $\bm{a}_\varepsilon\in\Omega^1(X)\otimes\R^3$ and 
\begin{align*}
\boldsymbol{\zeta}_\varepsilon
=\left(\sum_{\alpha=1}^3\zeta_{\varepsilon,1\alpha}\omega_{\varepsilon,\alpha},\, 
\sum_{\alpha=1}^3\zeta_{\varepsilon,2\alpha}\omega_{\varepsilon,\alpha},\, 
\sum_{\alpha=1}^3\zeta_{\varepsilon,3\alpha}\omega_{\varepsilon,\alpha}
\right)\in\Omega^2(X)\otimes\R^3,
\end{align*}
where $\zeta_{\varepsilon,\alpha\beta}\in\R$. 
Moreover, we have 
\begin{align*}
|d\bm{a}_\varepsilon|_{g_{\bdomega_\varepsilon}}&=O(\varepsilon^{3(\delta+2)/5}),\\
|\zeta_{\varepsilon\alpha\beta}|_{g_{\bdomega_\varepsilon}}&=O(\varepsilon^{(17-4\delta)/10}).
\end{align*}
Since we can take any negative number $\delta$ in $(-1/2,0)$, 
we put $\delta=-1/6$. 
It implies $|\widetilde{\boldsymbol{\omega}}_\varepsilon
-\boldsymbol{\omega}_\varepsilon|_{g_{\bdomega_\varepsilon}}=O(\varepsilon^{11/10})$. 
We also have $|Q_{\bdomega_\varepsilon}-I_3|=O(\varepsilon^{9/5})$ by the inequality $(5.8)$ in \cite{Foscolo2019}. 
By Lemma \ref{lem norm met}, we obtain the third inequality of $({\rm vi})$. 
\end{proof}

\section{Smooth maps}\label{sec map}
\subsection{Smooth maps from smooth manifold to $\T/\{ \pm 1\}$}
In this paper, for a smooth manifold $M$, a smooth map 
$f\colon M\to \T/\{ \pm 1\}$ is a morphism between orbifolds. 
Here, any smooth manifolds have naturally the structure of orbifolds. 
Moreover, any quotient space $M/\Gamma$ also has an orbifold structure, 
where $\Gamma$ is a finite group acting smoothly on $M$. 
Then $\T/\{ \pm 1\}$ is an orbifold of dimension $3$ and an open subset 
\begin{align*}
\overline{\T}_{{\rm reg}}&:=(\T/\{\pm 1\})\setminus\{ \bar{q}_1,\ldots,\bar{q}_8\}
\end{align*}
is a smooth manifold. 
Let $g_{\boldsymbol{\theta}}$ and $\delta_0>0$ be as in Section \ref{sec Foscolo}. 
Then a continuous map $f\colon M\to \T/\{ \pm 1\}$ is said to be smooth if 
there are index sets $A_j$ for $j=1,\ldots,8$ and an open cover $\{ U,U_{j,\alpha}|\, j=1,\ldots,8,\alpha\in A_j\}$ of $M$ such that 
the following hold. 
\begin{itemize}
\setlength{\parskip}{0cm}
\setlength{\itemsep}{0cm}
 \item[$({\rm i})$] $U=f^{-1}(\overline{\T}_{{\rm reg}})$, $U_{j,\alpha}\subset f^{-1}(B(\bar{q}_j,\delta_0))$.
 \item[$({\rm ii})$] $f|_U\colon U\to (\T/\{\pm 1\})\setminus\{ \bar{q}_1,\ldots,\bar{q}_8\}$ is a $C^\infty$-map.
 \item[$({\rm iii})$] There is a $C^\infty$- $\tilde{f}_{j,\alpha}\colon U_{j,\alpha}\to B(q_j,\delta_0)$ such that $\pi\circ \tilde{f}_{j,\alpha}=f|_{U_{j,\alpha}}$. 
 Here, $\pi\colon\T\to \T/\{\pm 1\}$ is the quotient map. 
\end{itemize}

In Section \ref{subsec GH}, 
we took parallel $1$-forms $\theta_1,\theta_2,\theta_3$ 
on $\T$. 
For $k\ge 0$, denote by $\Omega^k(\T)^{\{ \pm 1\}}$ 
the set of $k$-forms on $\T$ which are invariant by the pullback by the $\{ \pm 1\}$-action. 
Then for every $\Theta\in\Omega^k(\T)^{\{ \pm 1\}}$, 
$\Theta|_{\T\setminus \{ q_1,\ldots,q_8\}}$ descends to a differential $k$-form on the quotient manifold denoted by 
$\overline{\Theta}\in \Omega^k(\overline{\T}_{{\rm reg}})$. 

Let $f\colon M\to \T/\{ \pm 1\}$ be a smooth map 
and $\{ U,U_{j,\alpha}|\, j=1,\ldots,8,\alpha\in A_j\}$, $\tilde{f}_{j,\alpha}$ be associating datas satisfying $({\rm i}-{\rm iii})$. 
\begin{prop}
Let $\Theta\in\Omega^k(\T)^{\{ \pm 1\}}$. 
A differential $k$-from $f^*\Theta\in \Omega^k(M)$ given by 
\begin{align*}
f^*\Theta|_U&:=f^*\overline{\Theta},\\
f^*\Theta|_{U_{j,\alpha}}&:=\tilde{f}_{j,\alpha}^*\Theta
\end{align*}
is well-defined and determined uniquely by $f$ and $\Theta$. 
\label{prop pullback form}
\end{prop}
\begin{proof}
We need to show that two definitions of $f^*\Theta$ coincide on 
intersections $U\cap U_{j,\alpha}$, and it is independent of 
the choice of $U_{j,\alpha}$ and $\tilde{f}_{j,\alpha}$. 
Fix $j$ and let $U_1,U_2\subset f^{-1}(B(\bar{q}_j,\delta_0))$ 
be open sets, $\tilde{f}_i\colon U_i\to B(q_j,\delta_0)$ 
be smooth maps such that $\pi\circ \tilde{f}_i=f|_{U_i}$ 
for $i=1,2$. 
It suffices to show 
\begin{align*}
f^*\overline{\Theta}|_{U\cap U_1}
&=\tilde{f}_1^*\Theta|_{U\cap U_1},\\
\tilde{f}_1^*\Theta|_{U_1\cap U_2}
&=\tilde{f}_2^*\Theta|_{U_1\cap U_2}.
\end{align*}
By the definition of $\overline{\Theta}$, we can see 
$\pi^*\overline{\Theta}=\Theta$, therefore, we have 
\begin{align*}
f^*\overline{\Theta}|_{U\cap U_1}
=(\pi\circ\tilde{f}_1)^*\overline{\Theta}|_{U\cap U_1}
=\tilde{f}_1^*\Theta|_{U\cap U_1}.
\end{align*}
Consequently, we have 
\begin{align*}
\tilde{f}_1^*\Theta|_{U_1\cap U_2\cap U}
&=\tilde{f}_2^*\Theta|_{U_1\cap U_2\cap U}.
\end{align*}
Next, we take $p\in U_1\cap U_2\setminus U$. 
If $p$ is an interior point of $U_1\cap U_2\setminus U$, 
then we have $\tilde{f}_1\equiv\tilde{f}_2\equiv q_j$ on the neighborhood 
of $p$, hence $\tilde{f}_1^*\Theta=\tilde{f}_2^*\Theta$ at $p$. 
If we denote by $U_3$ the interior of $U_1\cap U_2\setminus U$, 
then we can see $\tilde{f}_1^*\Theta=\tilde{f}_2^*\Theta$ 
on $U_1\cap U_2\cap (U\cup U_3)$. 
Since $U_1\cap U_2\cap (U\cup U_3)$ is dense in $U_1\cap U_2$, we obtain $\tilde{f}_1^*\Theta|_{U_1\cap U_2}
=\tilde{f}_2^*\Theta|_{U_1\cap U_2}$ by 
the continuity of both hand sides. 
\end{proof}
\begin{prop}
Let $I\subset\R$ be an open interval containing the origin 
and $F\colon I\times M\to \T^3/\{ \pm 1\}$ a smooth map. 
For any closed form $\Theta\in\Omega^k(\T)^{\pm 1}$, 
$\frac{d}{dt}|_{t=0}F(t,\cdot)^*\Theta$ is an exact form on $M$. 
\label{prop var form}
\end{prop}
\begin{proof}
Put $f_t:=F(t,\cdot)$. 
Let $\hat{U},\hat{U}_{j,\alpha}$ be open sets of 
$I\times M$ such that $\hat{U}=F^{-1}(\overline{\T}_{{\rm reg}})$, 
$\hat{U}_{j,\alpha}\subset F^{-1}(B(\bar{q}_j,\delta_0))$ 
and 
$\{ 0\}\times M\subset \hat{U}\cup\bigcup_{j,\alpha}\hat{U}_{j,\alpha}$. 
Let $\tilde{F}_{j,\alpha}\colon\hat{U}_{j,\alpha}\to B(q_j,\delta_0)$ be smooth maps 
such that 
$\pi\circ\tilde{F}_{j,\alpha}=F|_{\hat{U}_{j,\alpha}}$. 

Fix $p\in M$. If $(0,p)\in \hat{U}$, then put 
\begin{align*}
w_p:=\left.\frac{d}{dt}\right|_{t=0}f_t(p)\in T_{f_0(p)}(\T/\{ \pm 1\}).
\end{align*}
If $(0,p)\in \hat{U}$, then $p\in f_0^{-1}(q_j)$ for some $j$ and there is $\alpha$ such that 
$(0,p)\in\hat{U}_{j,\alpha}$. We put 
\begin{align*}
\tilde{w}_p:=\left.\frac{d}{dt}\right|_{t=0}\tilde{F}_{j,\alpha}(t,p)\in T_{q_j}\T.
\end{align*}
Now, we take $U,U_{j,\alpha}$ such that 
\begin{align*}
(\{ 0\}\times M)\cap \hat{U}&=\{ 0\}\times U,\\
(\{ 0\}\times M)\cap \hat{U}_{j,\alpha}&=\{ 0\}\times U_{j,\alpha}
\end{align*}
and define $\eta\in\Omega^1(U)$ by 
\begin{align*}
\eta_p(v):=\overline{\Theta}_{f_0(p)}(w_p,df_0(v)),\quad v\in T_pU.
\end{align*}
Similarly, we define $\tilde{\eta}_{j,\alpha}\in\Omega^1(U_{j,\alpha})$ by 
\begin{align*}
(\tilde{\eta}_{j,\alpha})_p(v):=\Theta_{\tilde{f}_{j,\alpha,0}(p)}(\tilde{w}_p,d\tilde{f}_{j,\alpha,0}(v)),\quad v\in T_pU_{j,\alpha},
\end{align*}
where $\tilde{f}_{j,\alpha,t}=\tilde{F}_{j,\alpha}(t,\cdot)$. 
Now, we can extend $\eta$ to a smooth $1$-form on $M$ by 
$\eta|_{U_{j,\alpha}}:=\tilde{\eta}_{j,\alpha}$. 
By the argument similar to the proof of Proposition \ref{prop pullback form}, we can show that $\eta$ is well-defined and 
independent of the choice of 
$\hat{U}_{j,\alpha},U_{j,\alpha},\tilde{F}_{j,\alpha}$. 

Now, by the direct calculation, we can see 
\begin{align*}
\left.\frac{d}{dt}\right|_{t=0}f_t^*\overline{\Theta}|_{U}
&=d\eta,\\
\left.\frac{d}{dt}\right|_{t=0}\tilde{f}_{j,\alpha,t}^*\Theta|_{U_{j,\alpha}}
&=d\tilde{\eta}_{j,\alpha},
\end{align*}
which implies $\left.\frac{d}{dt}\right|_{t=0}f_t^*\Theta=d\eta$ on $M$. 
\end{proof}

For smooth maps $f_0,f_1\colon M\to \T/\{ \pm 1\}$, 
we say that $f_0$ and $f_1$ are homotopic 
if there is a smooth map $F\colon [0,1]\times M\to \T/\{\pm 1\}$ 
such that $F(0,\cdot)=f_0$ and $F(1,\cdot)=f_1$. 
In this case, we write $f_0\sim f_1$ and 
denote by $[f_0]$ the homotopy class represented by $f_0$. 

\begin{thm}
Let $M$ be a compact smooth manifold of dimension $m$ 
and $f\colon M\to \T/\{ \pm 1\}$ be a smooth map. 
Let $\omega\in\Omega^k(M)$ and $\Theta\in \Omega^{m-k}(\T)^{\{ \pm 1\}}$ be closed forms. 
Then the next quantity 
\begin{align*}
\int_M\omega\wedge f^*\Theta
\end{align*}
is determined by the homotopy class $[f]$. 
\label{thm homotopy inv}
\end{thm}
\begin{proof}
Let $f_0,f_1\in [f]$ and take a smooth map $F\colon [0,1]\times M\to\T/\{\pm 1\}$ such that $F(0,\cdot)=f_0$ and $F(1,\cdot)=f_1$. 
By Proposition \ref{prop var form}, we can see 
$\frac{d}{dt}F(t,\cdot)^*\Theta=d\eta_t$ for some 
$\eta_t\in\Omega^{k-1}(M)$. 
Since $d\omega=0$, then we have 
\begin{align*}
\frac{d}{dt}\int_M\omega\wedge F(t,\cdot)^*\Theta
&=\int_M\omega\wedge d\eta_t\\
&=(-1)^k\int_Md(\omega\wedge \eta_t)-(-1)^k\int_Md\omega\wedge \eta_t\\
&=0
\end{align*}
by the Stokes' Theorem. 
Therefore, we can see $\int_M\omega\wedge f_0^*\Theta=\int_M\omega\wedge f_1^*\Theta$. 
\end{proof}

For a smooth map 
$f\colon M\to N$ between smooth Riemannian manifolds 
$(M,g_M)$ and $(N,g_N)$, 
we put $\| df\|:=\sqrt{{\rm tr}_{g_M}(f^*g_N)}$. 

Let $g_{\boldsymbol{\theta}}$ be as in Section \ref{subsec GH}, 
it descends to a Riemannian metric on $\overline{T}_{{\rm reg}}$ 
which we also denote by $g_{\boldsymbol{\theta}}$. 
For a smooth map 
$f\colon M\to \T/\{ \pm 1\}$ with the associating data 
$\{ U,U_{j,\alpha},\tilde{f}_{j,\alpha}\}$, 
$\| df\|\colon M\to \R$ can be defined on $U$ 
since $f|_U$ is a smooth map between 
Riemannian manifolds $(U,g_M|_U)$ and $(\overline{T}_{{\rm reg}},g_{\boldsymbol{\theta}})$. 
Moreover, we extend $\| df\|$ to $M$ by 
\begin{align*}
\| df\|\big|_{U_{j,\alpha}}&:=\| d\tilde{f}_{j,\alpha}\|.
\end{align*}
Then by the argument similar to the proof of Proposition \ref{prop pullback form}, we can also show the well-definedness of $\| df\|$.

\begin{definition}
\normalfont
Let $(M,g)$ be a compact Riemannian manifold and 
$g_{\boldsymbol{\theta}}$ be a flat Riemannian metric on $\T$. 
For a smooth map $f\colon M\to \T/\{ \pm 1\}$, 
the Dirichlet energy $\E(f)$ is defined by 
\begin{align*}
\E(f)=\E(f,g,g_{\boldsymbol{\theta}}):=\int_M \| df\|^2d\vol_g,
\end{align*}
where $\vol_g$ is the volume measure of $g$. 
\end{definition}

\subsection{Smooth maps giving the collapsing to $\T/\{ \pm 1\}$}

Let $X$ be the $K3$ surface obtained by 
Theorem \ref{thm k3 to t3}. 
In this subsection, we construct a family of smooth maps $F_\varepsilon\colon X\to \T/\{ \pm 1\}$ by gluing. 
\begin{prop}
Let $(M,g,\boldsymbol{\eta},K,R,\phi,\nu)$ be an ALF gravitational instanton of type $A_{k-1}$ (resp. type $D_m$). 
Then there is a smooth map $\nu'\colon M\to \R^3$ (resp. $\nu'\colon M\to \R^3/\{\pm 1\}$) such that 
$\nu'|_{M\setminus K}=\nu$. Here, the smoothness of maps to $\R^3/\{\pm 1\}$ is defined similarly to the case of maps to $\T/\{\pm 1\}$. \label{prop ext hk moment ALF}
\end{prop}
\begin{rem}
\normalfont
In the above proposition, the smoothness of maps to $\R^3/\{\pm 1\}$ is defined similarly to the case of maps to $\T/\{\pm 1\}$. 
\end{rem}

\begin{proof}[Proof of Proposition \ref{prop ext hk moment ALF}]
In the case of type $A_{k-1}$, we can take $\nu'$ as the hyper-K\"ahler moment map associated with the Gibbons-Hawking construction of the ALF gravitational instanton of 
type $A_{k-1}$. 
Next, we consider the case of type $D_m$. 
First of all, assume $m=0$ for simplicity. Then $M$ is the Atiyah-Hitchin manifold constructed in \cite{AtiyahHitchin1988}. 
We follow the argument of \cite[Chapter 9]{AtiyahHitchin1988}. 
Put 
\[ 1_{\R^3}:=
\left (
\begin{array}{ccc}
1 & 0 & 0\\
0 & 1 & 0\\
0 & 0 & 1
\end{array}
\right ), e_1:=\left (
\begin{array}{ccc}
1 & 0 & 0\\
0 & -1 & 0\\
0 & 0 & -1
\end{array}
\right ), e_2:=\left (
\begin{array}{ccc}
-1 & 0 & 0\\
0 & -1 & 0\\
0 & 0 & 1
\end{array}
\right ),
\]
and $e_3:=e_1e_2$, which gives subgroups of $SO(3)$ defined by 
$G_0:=\{ 1_{\R^3},e_1\}$ and 
$G_1:=\{ 1_{\R^3},e_1,e_2,e_3\}$. 
Then there is a smooth function $\xi\colon M\to \R$ 
such that $\xi\ge 0$, $\xi^{-1}(0)\cong \R P^2$ 
and $\xi^{-1}(t)\cong SO(3)/G_1$ for $t>0$ as smooth manifolds. 
By these observations, we obtain the underlining manifold of $M$ as follows. 
Denote by $S^3\subset \quater$ the set of the quaternions of unit length, which is the Lie group isomorphic to $SU(2)$. 
Then the adjoint action ${\rm Ad}_g(x)=gxg^{-1}$ for $g\in S^3$ and $x\in {\rm Im}(\quater)=\R^3$ gives the double cover ${\rm Ad}\colon S^3\to SO(3)$. 
If we denote by $i,j,k$ the standard basis of ${\rm Im}(\quater)$, 
then the right action of the subgroup $D:=\{ \pm 1,\pm i,\pm j,\pm k\}$ 
gives the quotient space $S^3/D$ and ${\rm Ad}$ induces the isomorphism $S^3/D\cong SO(3)/G_1$. 
If we take a subgroup $\Z_4:=\{ \pm 1,\pm i\}\subset D$, then 
$S^3/\Z_4\cong SO(3)/G_0$ is a principal $U(1)$-bundle 
over $S^2$ by the action $[u]_{\Z_4}\cdot \lambda:=[u\lambda^{1/4}]_{\Z_4}$, where $[u]_{\Z_4}$ is the equivalent class replesented by $u\in S^3$ and $\lambda\in U(1)\subset \C$. The projection of the principal bundle is given by $[u]_{\Z_4}\mapsto {\rm Ad}_u(i)$. 
We obtain that the associate vector bundle $\tilde{E}:=(S^3/\Z_4)\times_{U(1)}\C$, which is the quotient space of $(S^3/\Z_4)\times\C$ by the equivalent relation $([u]_{\Z_4},z)\sim ([u\lambda^{1/4}]_{\Z_4},\lambda^{-1}z)$. 
The natural projection of the vector bundle is denoted by $\pi_{\tilde{E}}\colon\tilde{E}\to S^2$. 
Since $\tilde{E}$ has a free $\Z_2$-action defined by the involution $([u]_{\Z_4},z)\mapsto ([uj]_{\Z_4},\bar{z})$, then we obtain the quotient manifold $E:=\tilde{E}/\Z_2$, which is diffeomorphic to $M$. Here, the function $\xi$ is induced by $([u]_{\Z_4},z)\mapsto |z|^2$. 
Since $\tilde{E}$ is a vector bundle over $S^2$, we can take a fiber metric $h_{\tilde{E}}$ and the natural projection $\pi_{\tilde{E}}\colon\tilde{E}\to S^2$. Then the function on $\tilde{E}$ given by $\zeta\mapsto h_{\tilde{E}}(\zeta,\zeta)$ is smooth. 
Now, we regard $S^2$ as the unit sphere in $\R^3$ centered at the origin and denote by $I_{S^2}\colon S^2\to \R^3$ the inclusion map, which is smooth. 
Then we have a smooth map $\tilde{\mathcal{F}}\colon \tilde{E}\to \R^3$ by 
\begin{align*}
\tilde{\mathcal{F}}(\zeta):=h_{\tilde{E}}(\zeta,\zeta)\cdot I_{S^2}(\pi_{\tilde{E}}(\zeta)).
\end{align*}
If we denote by $\tau_{\tilde{E}}$ the involution on $\tilde{E}$, 
then we have 
\begin{align*}
h_{\tilde{E}}(\tau_{\tilde{E}}(\zeta),\tau_{\tilde{E}}(\zeta))
&=h_{\tilde{E}}(\zeta,\zeta),\\
\pi_{\tilde{E}}(\tau_{\tilde{E}}([u],z))
&={\rm Ad}_{uj}(i)=-{\rm Ad}_{u}(i)=-\pi_{\tilde{E}}(
[u],z),
\end{align*}
hence $\tilde{\mathcal{F}}\circ\tau_{\tilde{E}}=-\tilde{\mathcal{F}}$, 
which induces a smooth map $\mathcal{F}\colon M\to \R^3/\{ \pm 1\}$. 


Here, $\mathcal{F}|_{\del M}$ can be identified with $\nu_{D_0}|_{W^{-4}\times\{ 1\}}$, 
then we can extend $\nu$ to $M$ as a smooth map. 

Next, we consider the case of $m> 0$. 
By the argument in \cite[Remark 3.7]{Foscolo2019}, the underlining manifold of ALF gravitational instantons of type $D_{m}$ can be constructed by gluing Atiyah-Hitchin manifold and some ALF space of type $A_k$'s. 
Therefore, by combining the above arguments, we obtain a smooth extension of $\nu'$ to $M$. 
\end{proof}

\begin{prop}
Let $(\bm{m},\bm{k})\in\Z^{n+8}$, $(h,\theta)\in\mathcal{F}(S,\bm{m},\bm{k})$, $X,\{ U_\varepsilon,U^A_{i,\varepsilon},U^D_{j,\varepsilon}\}$ and 
 \begin{align*}
&(M^A_i,g^A_i,\boldsymbol{\eta}^A_i,K^A_i,R_0,\phi^A_i,\mu^A_i),\\
&(M^D_j,g^D_j,\boldsymbol{\eta}^D_j,K^D_j,R_0,\phi^D_j,\mu^D_j)
\end{align*}
be as in Theorem \ref{thm k3 to t3}. 
Then $\mu^A_i$ and $\mu^D_j$ can be extended to $M^A_i$, $M^D_j$, respectively, and there is a smooth map $F_\varepsilon\colon X\to \T/\{ \pm 1\}$ such that 
 \begin{align*}
F_\varepsilon|_{U_\varepsilon}
=\mu_{\varepsilon^{-1}+h},\quad
F_\varepsilon|_{U^A_{i,\varepsilon}}
=\varepsilon\mu^A_i,\quad
F_\varepsilon|_{U^D_{j,\varepsilon}}
=\varepsilon(1+\varepsilon\lambda_j)^{-1}\mu^D_j.
\end{align*}
Moreover, $[F_\varepsilon]=[F_{\varepsilon'}]$ for all sufficiently small $\varepsilon,\varepsilon'>0$. 
\label{prop glue Fe}
\end{prop}
\begin{proof}
Since the hyper-K\"ahler moment map $\mu_h$ is independent of the additive constants, we have $\mu_{\varepsilon^{-1}+h}=\mu_h$. 
By Proposition \ref{prop ext hk moment ALF}, we can extend $\mu^A_i$ and $\mu^D_j$ to $M^A_i$, $M^D_j$, respectively. 
By $({\rm iii})$ of Theorem \ref{thm k3 to t3}, 
we can see that $\mu_h$ and $\varepsilon\mu^A_i,\varepsilon(1+\varepsilon\lambda_j)^{-1}\mu^D_j$ are the same map on the intersections of their domain. Hence we obtain a smooth map $F_\varepsilon$ by gluing these maps. 
The homotopy equivalence $[F_\varepsilon]=[F_{\varepsilon'}]$ is obvious by the construction. 
\end{proof}

\section{Lower bounds of Dirichlet energy}\label{sec cal}
\subsection{Linear maps from $\R^4$ to $\R^3$}
We denote by 
$e_0,e_1,e_2,e_3\in \R^4$ and 
$f_1,f_2,f_3$ the orthonormal basis, 
$e^i\in(\R^4)^*$ and $f^j\in(\R^3)^*$ the dual basis. 
In the following argument, let 
$(i,j,k)=(1,2,3)$, $(2,3,1)$ or $(3,1,2)$. 
Put 
\begin{align*}
\omega_i&:=e^0\wedge e^i+e^j\wedge e^k,\\
\eta_i&:=f^j\wedge f^k,
\end{align*}
then $\boldsymbol{\omega}=(\omega_1,\omega_2,\omega_3)$ 
is a hyper-K\"ahler triple on $\R^4$. 
Conversely, if $\boldsymbol{\omega}=(\omega_1,\omega_2,\omega_3)$ 
is a hyper-K\"ahler triple on $\R^4$, then there is a basis 
$e_0,\ldots,e_4$ of $\R^4$ such that 
$\omega_i$ has the above description and 
$e_0,\ldots,e_4$ form an orthonormal basis for  $g_{\boldsymbol{\omega}}$. 
Let $A\in M_{3,4}(\R)$, put 
$A(e_l)=\sum_{h=1}^3A_l^hf_h$ and denote by $A^*\colon (\R^3)^*\to (\R^4)^*$ the adjoint of $A$. 
We consider a constant $\tau(A)$ defined by 
\begin{align*}
\sum_{l=1}^3\omega_l\wedge A^*\eta_l&=\tau(A)e^0\wedge e^1\wedge e^2\wedge e^3.
\end{align*}
Since 
\begin{align*}
\omega_i\wedge A^*\eta_i
&=(A^j_0A^k_i-A^j_iA^k_0+A^j_jA^k_k-A^j_kA^k_j)
e^0\wedge e^1\wedge e^2\wedge e^3,
\end{align*}
we have 
\begin{align*}
\tau(A)
&=-\frac{(A^2_0-A^3_1)^2
+(A^2_0+A^1_3)^2+(A^1_3+A^3_1)^2}{2}\\
&\quad\quad 
-\frac{(A^3_0-A^1_2)^2
+(A^3_0+A^2_1)^2+(A^1_2+A^2_1)^2}{2}\\
&\quad\quad 
-\frac{(A^1_0-A^2_3)^2
+(A^1_0+A^3_2)^2
+(A^2_3 +A^3_2)^2}{2}\\
&\quad\quad 
-\frac{(A^1_1-A^2_2)^2
+(A^2_2-A^3_3)^2
+(A^1_1-A^3_3)^2}{2}
+{\rm tr}(A^*A).
\end{align*}
Therefore, we obtain 
\begin{align}
\tau(A)\le {\rm tr}(A^*A)\label{ineq ptwise}
\end{align}
and the equality is satisfied iff 
\begin{align}
A^1_1=A^2_2=A^3_3,\, 
A^2_0=A^3_1=-A^1_3,\,
A^3_0=A^1_2=-A^2_1,\,
A^1_0=A^2_3=-A^3_2.\label{eq cal A}
\end{align}

\begin{lem}
Let $A\in M_{3,4}(\R)$. 
We have $\tau(A)={\rm tr}(A^*A)$ iff $A=0$ or 
\begin{align*}
\omega_i
&=\theta\wedge A^*f^i+D\cdot A^*f^j\wedge A^*f^k,
\end{align*}
for some $\theta\in (\R^4)^*$, 
where $D=\sqrt{{\rm tr}(A^*A)/3}$. 
\label{lem calibrate linear alg}
\end{lem}
\begin{proof}
Now we assume $\tau(A)={\rm tr}(A^*A)$ and put 
\begin{align*}
B&:=A^1_1=A^2_2=A^3_3,\quad
A^i:=A^i_0=A^j_k=-A^k_j,\\
\mu^i&:=A^*f^i
=A^ie^0+Be^i+A^ke^j-A^je^k,\\
\alpha&=Be^0-\sum_{l=1}^3A^le^l,
\end{align*}
then the representation matrix 
of 
\begin{align*}
(e^0,e^1,e^2,e^3)
\mapsto
(\alpha,A^*f^1,A^*f^2,A^*f^3)
\end{align*}
is given by 
\[ \tilde{A}=
\left (
\begin{array}{cccc}
B & A^1 & A^2 & A^3 \\
-A^1 & B & -A^3 & A^2\\
-A^2 & A^3 & B & -A^1\\
-A^3 & -A^2 & A^1 & B
\end{array}
\right ).
\]
If we put 
\begin{align*}
D:= B^2+\sum_{l=1}^3(A^l)^2,
\end{align*}
then 
\[ \tilde{A}^{-1}=D
\left (
\begin{array}{cccc}
B & -A^1 & -A^2 & -A^3 \\
A^1 & B & A^3 & -A^2\\
A^2 & -A^3 & B & A^1\\
A^3 & A^2 & -A^1 & B
\end{array}
\right ).
\]
Then we may write 
\begin{align*}
\omega_i
&=D^2\left( B\alpha +\sum_{l=1}^3A^l\mu^l\right)
\wedge \left( -A^i\alpha+B\mu^i -A^k\mu^j+A^j\mu^k\right)\\
&\quad\quad
+D^2\left( -A^j\alpha+B\mu^j -A^i\mu^k+A^k\mu^i\right)
\wedge \left( -A^k\alpha+B\mu^k -A^j\mu^i+A^i\mu^j\right)\\
&=D^2\left( B^2+\sum_{l=1}^3(A^l)^2\right)
(\alpha\wedge\mu^i+\mu^j\wedge\mu^k)\\
&=D(\alpha\wedge\mu^i+\mu^j\wedge\mu^k).
\end{align*}
By putting $\theta=D\alpha$, we have 
$\omega_i=\theta\wedge A^*f^i+D\cdot A^*f^j\wedge A^*f^k$. 
Conversely, if $A\neq 0$ and $\omega_i=\theta\wedge A^*f^i+D\cdot A^*f^j\wedge A^*f^k$, 
then $e^0=D^{-1/2}\theta,e^l=D^{1/2}A^*f^l$ form an orthonormal basis 
of $(\R^4)^*$ and the representation matrix of $A$ satisfies 
$A_i^i=D^{-1/2}$ and otherwise $A_i^j=0$, hence we have \eqref{eq cal A}. 
\end{proof}

\begin{thm}
Let $\hat{Y}\subset\T$, $h\in C^\infty(\hat{Y})$, $\hat{\boldsymbol{\eta}}_h$, $\hat{\mu}_h\colon \hat{M}_h\to\hat{Y}$, $\theta,\theta_1,\theta_2,\theta_3$ be as in Section \ref{subsec GH}. 
Then 
\begin{align*}
\sum_{i=1}^3\hat{\eta}_{h,i}\wedge\hat{\mu}_h^*(\theta_j\wedge\theta_k)
={\rm tr}_{g_{\hat{\boldsymbol{\eta}}_h}}(\hat{\mu}_h^*g_{\T}){\rm vol}_{g_{\hat{\boldsymbol{\eta}}_h}}.
\end{align*}
\label{thm calibrate gh}
\end{thm}
\begin{proof}
We fix $p\in\hat{M}_h$ and put 
$A:=(d\hat{\mu}_h)_p\colon T_p\hat{M}_h\to \R^3$. 
Applying Lemma \ref{lem calibrate linear alg} to $A$, 
we obtain the equality. 
\end{proof}

\begin{rem}
\normalfont
The above theorem says that 
the hyper-K\"ahler moment map associated with 
the Gibbons-Hawking ansatz is a calibrated map in the sense of 
\cite{hattori2024calibrated}. 
\end{rem}

\begin{thm}
Let $X$ be a smooth compact manifold of dimension $4$ 
and $\boldsymbol{\omega}=(\omega_1,\omega_2,\omega_3)$ be a triple of closed $2$-forms on $X$, 
$g_{\boldsymbol{\theta}}=\theta_1^2+\theta_2^2+\theta_3^2$ be as in Section \ref{subsec GH}. 
Let $f\colon X\to \T/\{ \pm 1\}$ be a smooth map. 
We put $\Theta_1:=\theta_2\wedge\theta_3$, $\Theta_2:=\theta_3\wedge\theta_1$, $\Theta_3:=\theta_1\wedge\theta_2$.
\begin{itemize}
\setlength{\parskip}{0cm}
\setlength{\itemsep}{0cm}
 \item[$({\rm i})$] For $i,j=1,2,3$, the following quantities 
\begin{align*}
\mathcal{I}_{ij}(f,\boldsymbol{\omega},\boldsymbol{\theta})
&:=\int_X\omega_i\wedge f^*(\Theta_j)
\end{align*}
are determined by the homotopy class $[f]$. 
 \item[$({\rm ii})$] If $\boldsymbol{\omega}$ is a hyper-K\"ahler triple on $X$, we have $\sum_{i=1}^3\mathcal{I}_{ii}(f,\boldsymbol{\omega},\boldsymbol{\theta})\le \E(f,g_{\boldsymbol{\omega}},g_{\boldsymbol{\theta}})$. 
\end{itemize}
\label{thm I<E}
\end{thm}
\begin{proof}
$({\rm i})$ 
Since $\theta_i\wedge\theta_j$ is a closed $2$-form on $\T$ 
which is invariant under the $\{ \pm 1\}$-action, 
then we have $\mathcal{I}_{ij}(f_0,\boldsymbol{\omega},\boldsymbol{\theta})=\mathcal{I}_{ij}(f_1,\boldsymbol{\omega},\boldsymbol{\theta})$ if $f_0$ is homotopic to $f_1$ by Theorem \ref{thm homotopy inv}. 

$({\rm ii})$ We have the pointwise inequality  
\begin{align*}
\frac{\omega_1\wedge f^*(\theta_2\wedge\theta_3)
+\omega_2\wedge f^*(\theta_3\wedge\theta_1)
+\omega_3\wedge f^*(\theta_1\wedge\theta_2)}{{\rm vol}_{g_{\boldsymbol{\omega}}}}
\le \| df\|^2
\end{align*}
by \eqref{ineq ptwise}. 
Integrating the both-hand-sides for the volume measure, 
we obtain $\sum_{i=1}^3\mathcal{I}_{ii}(f,\boldsymbol{\omega},\boldsymbol{\theta})\le \E(f,g_{\boldsymbol{\omega}},g_{\boldsymbol{\theta}})$. 
\end{proof}

\section{Estimates on the energy}\label{sec main estimate}
In this section, we give some estimates for $\mathcal{I}(F_\varepsilon,\tilde{\boldsymbol{\omega}}_\varepsilon,\boldsymbol{\theta}),\E(F_\varepsilon,g_{\tilde{\boldsymbol{\omega}}_\varepsilon},g_{\boldsymbol{\theta}})$ for the hyper-K\"ahler triples $\tilde{\boldsymbol{\omega}}_\varepsilon$ given by Theorem \ref{thm k3 to t3} and the smooth maps $F_\varepsilon$ given by Proposition \ref{prop glue Fe}.
This section aims to show 
\begin{align*}
\sum_{i=1}^3\mathcal{I}_{ii}(F_\varepsilon,\tilde{\boldsymbol{\omega}}_\varepsilon,\boldsymbol{\theta})>0,\quad
\lim_{\varepsilon\to 0}\frac{\E(F_\varepsilon,g_{\tilde{\boldsymbol{\omega}}_\varepsilon},g_{\boldsymbol{\theta}})}{\sum_{i=1}^3\mathcal{I}_{ii}(F_\varepsilon,\tilde{\boldsymbol{\omega}}_\varepsilon,\boldsymbol{\theta})}= 1.
\end{align*}
As we have shown in Theorem \ref{thm I<E}, 
$\E(f,g_{\tilde{\boldsymbol{\omega}}_\varepsilon},g_{\boldsymbol{\theta}}) / \sum_{i=1}^3\mathcal{I}_{ii}(F_\varepsilon,\tilde{\boldsymbol{\omega}}_\varepsilon,\boldsymbol{\theta})\ge 1$ always holds for any smooth map $f$ if the denominator is positive.

\subsection{Estimates on Riemannian manifold}
Here, we give some elementary estimates on general Riemannian manifolds and smooth maps. 
\begin{lem}
Let $(M,g)$ be an oriented Riemannian manifold of dimension $4$, $\alpha\in\Omega^k(M)$, 
$\beta\in\Omega^l(M)$. 
Assume that $\sup(|\alpha|_g|\beta|_g)$ and $\vol_g(M)$ are finite, where $\vol_g$ is the volume measure of $g$. 
Let $C$ be a positive constant appearing in Lemma \ref{lem norm asso}. 
Then we have $\left| \int_M\alpha\wedge\beta\right|\le C\vol_g(M)\sup(|\alpha|_g|\beta|_g)$. 
\end{lem}
\begin{proof}
By Lemma \ref{lem norm asso}, we have a pointwise inequality
\begin{align*}
-C|\alpha|_g|\beta|_gd\vol_g\le \alpha\wedge\beta \le C|\alpha|_g|\beta|_gd\vol_g.
\end{align*}
By integrating the above inequality, we have the assertion. 
\end{proof}

For a smooth map $f\colon M\to N$ between Riemannian manifolds $(M,g_M)$ and $(N,h_N)$, we have defined 
$\| df\|:=\sqrt{{\rm tr}_{g_M}(f^*g_N)}$. 
To emphasize the dependence on the metrics, we write 
$\| df\|_{g_M,g_N}=\| df\|$. 

\begin{lem}
Let $(M,g_M)$, $(N,g_N)$ be smooth Riemannian manifolds 
and $f\colon M\to N$ be a smooth map. 
Then 
$|f^*\alpha|_{g_M}\le \| df\|_{g_M,g_N}^kf^*(|\alpha|_{g_N})$ 
for any $\alpha\in\Omega^k(N)$. 
\label{lem map norm}
\end{lem}
\begin{proof}
Let $A\in M_{m,n}(\R)$ and $\alpha\in\Lambda^k(\R^m)^*$. 
Let $|\cdot|$ be the standard Euclidean norm on $\R^n,\R^m$. 
It suffices to show 
$|A^*\alpha|\le \sqrt{{\rm tr}({}^t\!AA)}^k|\alpha|$. 
There are orthonormal basis $\{ e_1,\ldots,e_n\}\subset\R^n$ and 
$\{ E_1,\ldots,E_m\}\subset\R^m$ and a constant $a_i\in\R$ for 
$i=1,\ldots,\min\{ m,n\}$ such that 
$A(e_i)=a_iE_i$ for $i\le m$ and $A(e_j)=0$ if $n>m$ and $j>m$. 
Let $\{ e^1,\ldots,e^n\}\subset(\R^n)^*$ and 
$\{ E^1,\ldots,E^m\}\subset(\R^m)^*$ be the dual basis. 
By using multi-index $I=(i_1,\ldots,i_k)\in \{ 1,\ldots,m\}^k$, 
we put $\alpha=\sum_I\alpha_I E^I$, where 
$E^I=E^{i_1}\wedge\cdots\wedge E^{i_k}$. 
Then we can see $|\alpha|^2=\sum_I|\alpha_I|^2$, $
{\rm tr}({}^t\!AA)=\sum_i a_i^2$, $A^*\alpha=\sum_I\alpha_Ia_{i_1}\cdots a_{i_k}e^I$ and 
\begin{align*}
|A^*\alpha|^2&=\sum_I|\alpha_Ia_{i_1}\cdots a_{i_k}|^2\\
&\le \sum_I|\alpha_I|^2{\rm tr}({}^t\!AA)^k
={\rm tr}({}^t\!AA)^k|\alpha|^2.
\end{align*}
\end{proof}

\begin{lem}
Let $(M,g_M)$, $(N,g_N)$ be smooth Riemannian manifolds 
and $f\colon M\to N$ be a smooth map. 
Let $g'_M$ be another Riemannian metric on $M$ such that 
$C_0g_M\le g'_M\le C_1g_M$ for some positive constants $C_0,C_1$. Then we have 
$C_1^{-1/2}\| df\|_{g_M,g_N}\le \| df\|_{g'_M,g_N}\le C_0^{-1/2}\| df\|_{g_M,g_N}$.
\label{lem E<E'}
\end{lem}
\begin{proof}
It is easy to see $\| df\|_{Cg_M,g_N}=C^{-1/2}\| df\|_{g_M,g_N}$ for $C>0$. 
Therefore, it suffices to show $g_M\le g'_M$ implies 
$\| df\|_{g'_M,g_N}\le \| df\|_{g_M,g_N}$. 
If $e_1,\ldots,e_n$ is an orthonormal basis of $(T_xX,g_M|_x)$, then we have $\| df\|_{g_M,g_N}^2|_x=\sum_{i=1}^n|df_x(e_i)|_{g_N}^2$. Since we can diagonalize $g'_M|_x$ by 
an orthonormal basis of $(T_xX,g_M|_x)$, we may choose 
$e_i$ such that $g_M'(e_i,e_j)=\lambda_i\delta_{ij}$ for some 
$\lambda_i>0$. Since $g_M\le g'_M$, we may suppose 
$\lambda_i\ge 1$. Hence we have 
\begin{align*}
\| df\|_{g'_M,g_N}^2|_x
=\sum_{i=1}^n\left| df_x\left( \frac{e_i}{\sqrt{\lambda_i}}\right)\right|_{g_N}^2
\le \sum_{i=1}^n\left| df_x\left( e_i\right)\right|_{g_N}^2
=\| df\|_{g_M,g_N}^2|_x.
\end{align*}
\end{proof}

\subsection{Estimates on the energies and the integral invariants}
In this subsection, we inherit the setting and the notation of Section \ref{sec Foscolo}. 
Let $(X,\bdomega_\varepsilon,\tilde{\bdomega}_\varepsilon,F_\varepsilon)$ be as in Theorem \ref{thm k3 to t3} and Proposition \ref{prop glue Fe}. 
For $\varepsilon>0$ and $i,j=1,2,3$, put 
\begin{align*}
\tilde{\mathcal{I}}_{\varepsilon,ij}
&:=\mathcal{I}_{ij}( F_\varepsilon,\tilde{\boldsymbol{\omega}}_\varepsilon,\boldsymbol{\theta}),\quad
\mathcal{I}_{\varepsilon,ij}
:=\mathcal{I}_{ij}( F_\varepsilon,\boldsymbol{\omega}_\varepsilon,\boldsymbol{\theta}),\\
\tilde{\E}_\varepsilon
&:=\E( F_\varepsilon,g_{\tilde{\boldsymbol{\omega}}_\varepsilon},g_{\boldsymbol{\theta}}),\quad
\E_\varepsilon
:=\E( F_\varepsilon,g_{\boldsymbol{\omega}_\varepsilon},g_{\boldsymbol{\theta}}).
\end{align*}
For a subset $U\subset X$, we write 
\begin{align*}
\tilde{\mathcal{I}}_{\varepsilon,ij}|_U
&:=\tilde{\mathcal{I}}_{ij}(F_\varepsilon|_U,\tilde{\boldsymbol{\omega}}_\varepsilon|_U,\boldsymbol{\theta}),\quad
\mathcal{I}_{\varepsilon,ij}|_U
:=\mathcal{I}_{ij}( F_\varepsilon|_U,\boldsymbol{\omega}_\varepsilon|_U,\boldsymbol{\theta}),\\
\tilde{\E}_\varepsilon|_U
&:=\E( F_\varepsilon|_U,g_{\tilde{\boldsymbol{\omega}}_\varepsilon}|_U,g_{\boldsymbol{\theta}}),\quad
\E_\varepsilon|_U
:=\E( F_\varepsilon|_U,g_{\boldsymbol{\omega}_\varepsilon}|_U,g_{\boldsymbol{\theta}}).
\end{align*}
Define matrices $\tilde{\mathcal{I}}_{\varepsilon},\mathcal{I}_{\varepsilon},\tilde{\mathcal{I}}_{\varepsilon}|_U,\mathcal{I}_{\varepsilon}|_U \in M_3(\R)$ by 
\begin{align*}
\tilde{\mathcal{I}}_{\varepsilon}&:=( \tilde{\mathcal{I}}_{\varepsilon,ij})_{i,j=1,2,3},\quad
\tilde{\mathcal{I}}_{\varepsilon}|_U:=(\tilde{\mathcal{I}}_{\varepsilon,ij}|_U)_{i,j=1,2,3},\\
\mathcal{I}_{\varepsilon}&:=(\mathcal{I}_{\varepsilon,ij})_{i,j=1,2,3},\quad
\mathcal{I}_{\varepsilon}|_U :=(\mathcal{I}_{\varepsilon,ij}|_U)_{i,j=1,2,3},
\end{align*}
and put  
\begin{align*}
U'_\varepsilon:=M_{\varepsilon^{-1}+h}|_{\ge 2\varepsilon^{2/5}} (\subset U_\varepsilon).
\end{align*}

\begin{prop}
For sufficiently small $\varepsilon>0$, we have 
\begin{align*}
\mathcal{I}_{\varepsilon}|_{U'_\varepsilon}
&=\frac{\varepsilon}{3}\left\{ 3\pi\vol_{g_{\boldsymbol{\theta}}}(\T)-64\pi^2(n+4)\varepsilon^{6/5}\right\}\cdot I_3,\\
\mathcal{E}_\varepsilon|_{U'_\varepsilon}
&={\rm tr}(\mathcal{I}_{\varepsilon}|_{U'_\varepsilon})
=\left\{ 3\pi\vol_{g_{\boldsymbol{\theta}}}(\T)-64\pi^2(n+4)\varepsilon^{6/5}\right\}\varepsilon.
\end{align*}
\label{prop I=E=e}
\end{prop}
\begin{proof}
By $({\rm ii})$ of Theorem \ref{thm k3 to t3} and the definition of 
$\hat{\bdeta}_{\varepsilon^{-1}+h}$, we can see 
\begin{align*}
\mathcal{I}_{\varepsilon,ii}|_{U'_\varepsilon}
&=\frac{1}{2}\int_{\tau^{-1}(U'_\varepsilon)}\varepsilon\hat{\eta}_{\varepsilon^{-1}+h,i}\wedge \hat{\mu}_{\varepsilon^{-1}+h}^*\Theta_i\\
&=\frac{\varepsilon}{2}\int_{\tau^{-1}(U'_\varepsilon)}\theta_1\wedge \theta_2\wedge\theta_3\wedge\theta,
\end{align*}
and similarly, we have $\mathcal{I}_{\varepsilon,ii}|_{U'_\varepsilon}=0$ if $i\neq j$. 
Here, the orientation of $X$ is determined such that the integration of $\theta_1\wedge \theta_2\wedge\theta_3\wedge\theta$ is positive. 
If we denote by $B_\T(S,r)$ the $r$-ball in $(\T,g_{\boldsymbol{\theta}})$ of the $(2n+8)$-points set $S=\{ q_j,p_i,-p_i\}$, 
then by Fubini's Theorem, we have 
\begin{align*}
\mathcal{I}_{\varepsilon,ii}|_{U'_\varepsilon}
&=\pi\varepsilon\vol_{g_{\boldsymbol{\theta}}}\left( \T\setminus B_\T(S,2\varepsilon^{2/5})\right)\\
&=\left\{ \pi\vol_{g_{\boldsymbol{\theta}}}(\T)-\frac{64\pi^2(n+4)\varepsilon^{6/5}}{3}\right\}\varepsilon.
\end{align*}
Moreover, since the pair $(\bdomega_\varepsilon|_{U'_\varepsilon},F_\varepsilon|_{U'_\varepsilon})$ is locally given by the rescaling of the Gibbons Hawking ansatz, then by Theorem \ref{thm calibrate gh}, we can see 
${\rm tr}(\mathcal{I}_\varepsilon|_{U'_\varepsilon})=\mathcal{E}_\varepsilon|_{U'_\varepsilon}$.
\end{proof}

\begin{lem}
We have 
\begin{align*}
\vol_{g_{\bdomega_\varepsilon}}(U_\varepsilon')
&=\pi\varepsilon\left( 
\vol_{g_{\boldsymbol{\theta}}}( \T)
+\varepsilon\int_\T h\,\vol_{g_{\boldsymbol{\theta}}}
+O(\varepsilon^{6/5})
\right),\\
\vol_{g^A_i}(U^A_{i,\varepsilon})&=O(\varepsilon^{-9/5}),\\
\vol_{g^D_j}(U^D_{j,\varepsilon})&=O(\varepsilon^{-9/5}),\\
\vol_{g_{\bdomega_\varepsilon}}(X)
&=\pi\varepsilon\left( 
\vol_{g_{\boldsymbol{\theta}}}( \T)
+\varepsilon\int_\T h\,\vol_{g_{\boldsymbol{\theta}}}
+O(\varepsilon^{6/5})\right).
\end{align*}
\label{lem vol estimate}
\end{lem}
\begin{proof}
We can see 
\begin{align*}
\vol_{g_{\bdomega_\varepsilon}}(U_\varepsilon')
&=\frac{1}{2}\int_{\hat{M}_{\varepsilon^{-1}+h}|_{\ge 2\varepsilon^{2/5}}}\varepsilon^2\vol_{g_{\hat{\bdeta}_{\varepsilon^{-1}+h}}}\\
&=\pi\varepsilon^2\int_{\T\setminus B_\T(S,2\varepsilon^{2/5})}(\varepsilon^{-1}+h)\vol_{g_{\boldsymbol{\theta}}}\\
&=\pi\varepsilon\left( 
\vol_{g_{\boldsymbol{\theta}}}( \T)
-\frac{64(n+4)\pi\varepsilon^{6/5}}{3}
\right)
 + \pi\varepsilon^2\int_{ \T\setminus B_\T(S,2\varepsilon^{2/5})}h\,\vol_{g_{\boldsymbol{\theta}}}.
\end{align*}
Since 
\begin{align*}
\int_{ \T\setminus B_{\T}(S,2\varepsilon^{2/5})}h\,\vol_{g_{\boldsymbol{\theta}}}
&=\int_{\T}h\,\vol_{g_{\boldsymbol{\theta}}}
-2\sum_{i=1}^n \int_{B_{\T}(p_i,2\varepsilon^{2/5})}h\,\vol_{g_{\boldsymbol{\theta}}}\\
&\quad\quad
-\sum_{j=1}^8 \int_{B_{\T}(q_j,2\varepsilon^{2/5})}h\,\vol_{g_{\boldsymbol{\theta}}}\\
&=\int_{\T}h\,\vol_{g_{\boldsymbol{\theta}}}+O(\varepsilon^{4/5}),
\end{align*}
we have 
\begin{align*}
\vol_{g_{\bdomega_\varepsilon}}(U_\varepsilon')
&=\pi\varepsilon\left( 
\vol_{g_{\boldsymbol{\theta}}}( \T)
+\varepsilon\int_\T h\,\vol_{g_{\boldsymbol{\theta}}}
+O(\varepsilon^{6/5})\right).
\end{align*}
Recall that $U^A_{i,\varepsilon}=M^A_i|_{<2\varepsilon^{-3/5}}=(\mu^A_i)^{-1}(B_{\R^3}(0,2\varepsilon^{-3/5}))$. 
Since the volume of the inverse image of $\mu^A_i$ is proportional to the Euclidean volume of its image by $\mu^A_i$, we have 
$\vol_{g^A_i}(U^A_{i,\varepsilon})=O(\varepsilon^{-9/5})$. 
The estimates on $U^D_{j,\varepsilon}$ is proved in the same way. 

By the estimates in $({\rm vi})$ of Theorem \ref{thm k3 to t3}, 
we can see $\vol_{g_{\bdomega_\varepsilon}}(U^A_{i,\varepsilon})
\le \varepsilon^4(1+C\varepsilon)^2\vol_{g^A_i}(U^A_{i,\varepsilon})=O(\varepsilon^{11/5})$ 
and 
$\vol_{g_{\bdomega_\varepsilon}}(U^D_{j,\varepsilon})=O(\varepsilon^{11/5})$, hence we have 
\begin{align*}
\vol_{g_{\bdomega_\varepsilon}}(X)
=\pi\varepsilon\left( 
\vol_{g_{\boldsymbol{\theta}}}( \T)
+\varepsilon\int_\T h\,\vol_{g_{\boldsymbol{\theta}}}
+O(\varepsilon^{6/5})\right).
\end{align*}
\end{proof}


\begin{prop}
We have 
\begin{align*}
\max\left\{ \left| \tilde{\mathcal{I}}_{\varepsilon,kl}|_{X\setminus U'_\varepsilon}\right|,\,
\left| \mathcal{I}_{\varepsilon,kl}|_{X\setminus U'_\varepsilon}\right|,\,
\tilde{\E}_\varepsilon|_{X\setminus U'_\varepsilon},\,
\E_\varepsilon|_{X\setminus U'_\varepsilon} \right\}
=O(\varepsilon^{11/5}).
\end{align*}
\label{prop IE on small region}
\end{prop}
\begin{proof}
Let $k,l=1,2,3$. 
By Lemma \ref{lem map norm}, we can see 
\begin{align*}
|\omega_{\varepsilon,k}\wedge F_\varepsilon^*(\Theta_l)|_{g_{\bdomega_\varepsilon}}
&\le 2\| dF_\varepsilon\|_{g_{\bdomega_\varepsilon},g_{\boldsymbol{\theta}}}^2,\\
|\tilde{\omega}_{\varepsilon,k}\wedge F_\varepsilon^*(\Theta_l)|_{g_{\tilde{\bdomega}_\varepsilon}}
&\le 2\| dF_\varepsilon\|_{g_{\tilde{\bdomega}_\varepsilon},g_{\boldsymbol{\theta}}}^2.
\end{align*}
Then we obtain 
\begin{align*}
\left| \mathcal{I}_{\varepsilon,kl}|_{X\setminus U'_\varepsilon}\right|
&\le 2\E_\varepsilon|_{X\setminus U'_\varepsilon}
,\quad
\left| \tilde{\mathcal{I}}_{\varepsilon,kl}|_{X\setminus U'_\varepsilon}\right|
\le 2\tilde{\E}_\varepsilon|_{X\setminus U'_\varepsilon}.
\end{align*}
By $({\rm vi})$ of Theorem \ref{thm k3 to t3}, 
we have $C^{-1}g_{\tilde{\bdomega}_\varepsilon}\le g_{\bdomega_\varepsilon}\le Cg_{\tilde{\bdomega}_\varepsilon}$ for some $C\ge 1$, then Lemma \ref{lem E<E'} implies 
$\| dF_\varepsilon\|_{g_{\tilde{\bdomega}_\varepsilon},g_{\boldsymbol{\theta}}}^2
\le C\| dF_\varepsilon\|_{g_{\bdomega_\varepsilon},g_{\boldsymbol{\theta}}}^2$. 
We also have $\vol_{g_{\tilde{\bdomega}_\varepsilon}}\le C^2\vol_{g_{\bdomega_\varepsilon}}$, 
consequently, we obtain 
\begin{align}
\tilde{\E}_\varepsilon|_{X\setminus U'_\varepsilon}
\le C^3\E_\varepsilon|_{X\setminus U'_\varepsilon}.\label{ineq tild E < C E}
\end{align}
Therefore, giving the upper bound of $\E_\varepsilon|_{X\setminus U'_\varepsilon}$ suffices. 

By $({\rm i,iii})$ of Theorem \ref{thm k3 to t3}, we may write 
\begin{align*}
X\setminus U'_\varepsilon
=\bigcup_{i=1}^nU^A_{i,\varepsilon}\cup\bigcup_{j=1}^8U^D_{j,\varepsilon}.
\end{align*}
By $({\rm vi})$ of Theorem \ref{thm k3 to t3}, 
we obtained $\varepsilon^2C^{-1}g_i^A\le g_{\bdomega_\varepsilon}\le \varepsilon^2Cg_i^A$ for some $C>0$. 
By a similar argument to show \eqref{ineq tild E < C E}, we can see 
\begin{align*}
\E_\varepsilon|_{U^A_{i,\varepsilon}}
\le C^3\E( F_\varepsilon|_{U^A_{i,\varepsilon}},\varepsilon^2 g^A_i|_{U^A_{i,\varepsilon}},g_{\boldsymbol{\theta}})
=C^3\varepsilon^2\E( F_\varepsilon|_{U^A_{i,\varepsilon}},g^A_i|_{U^A_{i,\varepsilon}},g_{\boldsymbol{\theta}}).
\end{align*}
Recall that $F_\varepsilon|_{U^A_{i,\varepsilon}}=\varepsilon\mu^A_i$ by Proposition \ref{prop glue Fe}. 
Accordingly, we have 
\begin{align*}
\E( F_\varepsilon|_{U^A_{i,\varepsilon}},g^A_i|_{U^A_{i,\varepsilon}},g_{\boldsymbol{\theta}})
=\int_{U^A_{i,\varepsilon}}\varepsilon^2\| d\mu^A_i\|_{g^A_i,g_{\boldsymbol{\theta}}}^2d\vol_{g^A_i}
\end{align*}
Here, the map $\mu^A_i$ is close to a Riemannian submersion to $\R^3$ with the flat metric and the length of fibers is close to $2\pi$ as $|\mu^A_i|\to\infty$. In particular, $\| d\mu^A_i\|$ is bounded on ALF space $M^A_i$. By taking $C$ larger if necessary, we have 
\begin{align*}
\E( F_\varepsilon|_{U^A_{i,\varepsilon}},g^A_i|_{U^A_{i,\varepsilon}},g_{\boldsymbol{\theta}})
\le C\varepsilon^2\vol_{g^A_i}(U^A_{i,\varepsilon}). 
\end{align*}
Hence we have 
\begin{align*}
\E( F_\varepsilon|_{U^A_{i,\varepsilon}},g^A_i|_{U^A_{i,\varepsilon}},g_{\boldsymbol{\theta}})
= O(\varepsilon^{11/5})
\end{align*}
by Lemma \ref{lem vol estimate}. 
The estimates on $U^D_{j,\varepsilon}$ is similar. 
\end{proof}

\begin{prop}
We have 
\begin{align*}
\tilde{\mathcal{I}}_\varepsilon
-\mathcal{I}_\varepsilon
&=O(\varepsilon^{27/10}),\\
\tilde{\E}_\varepsilon|_{U'_\varepsilon}-\E_\varepsilon|_{U'_\varepsilon}
&=O(\varepsilon^{21/10}).
\end{align*}
\label{prop error IE}
\end{prop}
\begin{proof}
By the argument in the proof of $({\rm vi})$ of Theorem \ref{thm k3 to t3}, we have $\tilde{\bdomega}_\varepsilon=\bdomega
_\varepsilon+d\bm{a}_\varepsilon+\boldsymbol{\zeta}_\varepsilon$ 
such that $\boldsymbol{\zeta}_\varepsilon=(\sum_{l=1}^3\zeta_{\varepsilon,kl}\omega_{\varepsilon,l})_{k=1,2,3}$ 
and $|\zeta_{\varepsilon,kl}|=O(\varepsilon^{(17-4\delta)/10})$. Here, $\delta$ was taken from $(-1/2,0)$, we may suppose 
$|\zeta_{\varepsilon,kl}|=O(\varepsilon^{17/10})$. 

Now, it is easy to see that the value of $\mathcal{I}_{kl}(f,\bdomega,\boldsymbol{\theta})$ is determined by the cohomology class of $\bdomega$. 
Since we have the identity between the cohomology classes 
$[\tilde{\bdomega}_\varepsilon]=[\bdomega
_\varepsilon]+[\boldsymbol{\zeta}_\varepsilon]$, we obtain 
\begin{align*}
\tilde{\mathcal{I}}_\varepsilon
&=(\mathcal{I}_{kl}(f,\bdomega_\varepsilon+\boldsymbol{\zeta}_\varepsilon,\boldsymbol{\theta}))_{k,l}\\
&=(I_3
+\tilde{\zeta}_{\varepsilon})\mathcal{I}_{\varepsilon},
\end{align*}
where $\tilde{\zeta}_{\varepsilon}=(\zeta_{\varepsilon,kl})_{k,l}\in M_3(\R)$. 
By Propositions \ref{prop I=E=e} and \ref{prop IE on small region}, we have 
\begin{align*}
\mathcal{I}_{\varepsilon}
=\mathcal{I}_{\varepsilon}|_{U'_\varepsilon}+\mathcal{I}_{\varepsilon}|_{X\setminus U'_\varepsilon}
=O(\varepsilon), 
\end{align*}
hence 
\begin{align*}
\tilde{\mathcal{I}}_\varepsilon-\mathcal{I}_{\varepsilon}
&=\tilde{\zeta}_{\varepsilon}\mathcal{I}_{\varepsilon}
=O(\varepsilon^{27/10}).
\end{align*}

By $({\rm vi})$ of Theorem \ref{thm k3 to t3} and 
Lemma \ref{lem map norm}, there is a constant $C>0$ such that 
$(1+C\varepsilon^{11/10})^{-1}\| dF_\varepsilon\|_{g_{\tilde{\bdomega}_\varepsilon},g_{\boldsymbol{\theta}}}^2\le \| dF_\varepsilon\|_{g_{\bdomega_\varepsilon},g_{\boldsymbol{\theta}}}^2\le (1-C\varepsilon^{11/10})^{-1}\| dF_\varepsilon\|_{g_{\tilde{\bdomega}_\varepsilon},g_{\boldsymbol{\theta}}}^2$. 
We also have 
$(1-C\varepsilon^{11/10})^2\vol_{g_{\tilde{\bdomega}_\varepsilon}}\le \vol_{g_{\bdomega_\varepsilon}}\le (1+C\varepsilon^{11/10})^2\vol_{g_{\tilde{\bdomega}_\varepsilon}}$. 
Then we have 
\begin{align*}
\frac{(1-C\varepsilon^{11/10})^2}{1+C\varepsilon^{11/10}}
\le
\frac{\E_\varepsilon|_{U'_\varepsilon}}{\tilde{\E}_\varepsilon|_{U'_\varepsilon}}\le \frac{(1+C\varepsilon^{11/10})^2}{1-C\varepsilon^{11/10}}, 
\end{align*}
which gives 
\begin{align*}
|\E_\varepsilon|_{U'_\varepsilon}-\tilde{\E}_\varepsilon|_{U'_\varepsilon}|
=\E_\varepsilon|_{U'_\varepsilon}\cdot O(\varepsilon^{11/10})=O(\varepsilon^{21/10})
\end{align*}
by Proposition \ref{prop I=E=e}. 
\end{proof}

\begin{thm}
We have 
\begin{align*}
\tilde{\mathcal{I}}_\varepsilon
&=\pi\varepsilon\vol_{g_{\boldsymbol{\theta}}}(\T)\cdot I_3+O(\varepsilon^{11/5}),\\
\frac{\tilde{\E}_\varepsilon}{{\rm tr}(\tilde{\mathcal{I}}_\varepsilon)}
&=1+O(\varepsilon^{11/10}).
\end{align*}
\label{thm main estimate}
\end{thm}
\begin{proof}
We have 
\begin{align*}
\tilde{\mathcal{I}}_\varepsilon - \mathcal{I}_\varepsilon|_{U'_\varepsilon}
&= \left( \tilde{\mathcal{I}}_\varepsilon - \mathcal{I}_\varepsilon\right)
+\mathcal{I}_\varepsilon|_{X\setminus U'_\varepsilon}=O(\varepsilon^{11/5}),\\
\tilde{\mathcal{E}}_\varepsilon - \mathcal{E}_\varepsilon|_{U'_\varepsilon}
&= \tilde{\mathcal{E}}_\varepsilon|_{X\setminus U'_\varepsilon}
+ \left( \tilde{\mathcal{E}}_\varepsilon|_{U'_\varepsilon} - \mathcal{E}_\varepsilon|_{U'_\varepsilon}\right)
=O(\varepsilon^{21/10}).
\end{align*}
By Propositions \ref{prop I=E=e}, \ref{prop IE on small region} and \ref{prop error IE}, we have 
\begin{align*}
\tilde{\mathcal{I}}_\varepsilon - \pi\varepsilon\vol_{g_{\boldsymbol{\theta}}}(\T)\cdot I_3
&= O(\varepsilon ^{11/5}),\\
\tilde{\mathcal{E}}_\varepsilon - 3\pi\varepsilon\vol_{g_{\boldsymbol{\theta}}}(\T)
&= O(\varepsilon ^{21/10}).
\end{align*}
Since ${\rm tr}(\tilde{\mathcal{I}}_\varepsilon)=3\pi\varepsilon\vol_{g_{\boldsymbol{\theta}}}(\T)+O(\varepsilon ^{11/5})$, we have 
$\tilde{\mathcal{E}}_\varepsilon/{\rm tr}(\tilde{\mathcal{I}}_\varepsilon)=1+O(\varepsilon ^{11/10})$. 
\end{proof}
\begin{proof}[Proof of Theorem \ref{thm main intro}]
Put $g_\varepsilon=g_{\tilde{\bdomega}_\varepsilon}$. 
The invariant $\mathcal{I}_\varepsilon([f])$ in Theorem \ref{thm main intro} is given by $\sum_{i=1}^3\mathcal{I}_{ii}(f,\tilde{\bdomega}_\varepsilon,\boldsymbol{\theta})$ in  Theorem \ref{thm I<E}. 
Then we obtain the inequality $\E(f,g_\varepsilon,g_\T)\ge \mathcal{I}_\varepsilon([f])$ for all $f\colon X\to \T/\{ \pm 1\}$. 

$({\rm i})$ is given by \cite[Theorem 6.17]{Foscolo2019}. 
$({\rm ii})$ was shown by Proposition \ref{prop glue Fe}. 
Now, $\mathcal{I}_\varepsilon([F_\varepsilon])$ is equal to ${\rm tr}(\tilde{I}_\varepsilon)$ in this section. Then $({\rm iii,iv})$ is shown by Theorem \ref{thm main estimate}. 
\end{proof}

We also obtain the estimate on the difference between $\tilde{\mathcal{I}}_\varepsilon$ and the volume of $g_{\tilde{\bdomega}_\varepsilon}$ as follows. 

\begin{thm}
We have 
\begin{align*}
\frac{\vol_{g_{\tilde{\bdomega}_\varepsilon}}(X)}{{\rm tr}(\tilde{\mathcal{I}}_\varepsilon)}
&=\frac{1}{3}+\frac{\varepsilon \int_\T h\,\vol_{g_{\boldsymbol{\theta}}}}{3\vol_{g_{\boldsymbol{\theta}}}(\T) }+O(\varepsilon^{11/10}).
\end{align*}
In particular, $3\vol_{g_{\tilde{\bdomega}_\varepsilon}}(X)/{\rm tr}(\tilde{\mathcal{I}}_\varepsilon)\to 1$ as $\varepsilon\to 0$, however, $3\vol_{g_{\tilde{\bdomega}_\varepsilon}}(X)\neq {\rm tr}(\tilde{\mathcal{I}}_\varepsilon)$ for sufficiently small $\varepsilon$ unless $\int_\T h\,\vol_{g_{\boldsymbol{\theta}}}=0$. 
\label{thm I and vol}
\end{thm}
\begin{proof}
By $({\rm vi})$ of Theorem \ref{thm k3 to t3}, 
we have $(1+C\varepsilon^{11/10})^{-2}\vol_{g_{\bdomega_\varepsilon}}\le 
\vol_{g_{\tilde{\bdomega}_\varepsilon}}\le
(1-C\varepsilon^{11/10})^{-2}\vol_{g_{\bdomega_\varepsilon}}$. 
By Lemma \ref{lem vol estimate}, we can see 
\begin{align*}
\left|\vol_{g_{\tilde{\bdomega}_\varepsilon}}(X)-\pi\varepsilon\vol_{g_{\boldsymbol{\theta}}}(\T)-\pi\varepsilon^2\int_\T h\,\vol_{g_{\boldsymbol{\theta}}}\right|
&\le\left|\vol_{g_{\tilde{\bdomega}_\varepsilon}}(X)-\vol_{g_{\bdomega_\varepsilon}}(X)\right|
+O(\varepsilon^{11/5})\\
&=\vol_{g_{\bdomega_\varepsilon}}(X)\cdot O(\varepsilon^{11/10})+O(\varepsilon^{11/5})\\
&=O(\varepsilon^{21/10}).
\end{align*}
By Theorem \ref{thm main estimate}, we have 
 \begin{align*}
\frac{\vol_{g_{\tilde{\bdomega}_\varepsilon}}}{{\rm tr}(\tilde{\mathcal{I}}_\varepsilon)}
&=\frac{\pi\varepsilon\vol_{g_{\boldsymbol{\theta}}}(\T)+\pi\varepsilon^2\int_\T h\,\vol_{g_{\boldsymbol{\theta}}} + O(\varepsilon^{21/10})}{3\pi\varepsilon\vol_{g_{\boldsymbol{\theta}}}(\T)+O(\varepsilon^{11/5})}\\
&=\frac{1+\varepsilon \int_\T h\,\vol_{g_{\boldsymbol{\theta}}} / \vol_{g_{\boldsymbol{\theta}}}(\T) + O(\varepsilon^{11/10})}{3+O(\varepsilon^{6/5})}. 
\end{align*}
\end{proof}

\bibliographystyle{plain}

\begin{thebibliography}{10}

\bibitem{AtiyahHitchin1988}
Michael Atiyah and Nigel Hitchin.
\newblock {\em The geometry and dynamics of magnetic monopoles}.
\newblock M. B. Porter Lectures. Princeton University Press, Princeton, NJ,
  1988.

\bibitem{Foscolo2019}
Lorenzo Foscolo.
\newblock A{LF} gravitational instantons and collapsing {R}icci-flat metrics on
  the {$K3$} surface.
\newblock {\em J. Differential Geom.}, 112(1):79--120, 2019.

\bibitem{GH1978}
Gary~W Gibbons and Stephen~W Hawking.
\newblock Gravitational multi-instantons.
\newblock {\em Physics Letters B}, 78(4):430--432, 1978.

\bibitem{GW2000}
Mark Gross and P.~M.~H. Wilson.
\newblock Large complex structure limits of ${K}3$ surfaces.
\newblock {\em J. Differential Geom.}, 55(3):475--546, 2000.

\bibitem{hattori2024calibrated}
Kota Hattori.
\newblock Smooth maps minimizing the energy and the calibrated geometry.
\newblock {\em J. Geom. Anal.}, 34(1):Paper No. 9, 22, 2024.

\bibitem{HSVZ2022}
Hans-Joachim Hein, Song Sun, Jeff Viaclovsky, and Ruobing Zhang.
\newblock Nilpotent structures and collapsing {R}icci-flat metrics on the {K}3
  surface.
\newblock {\em J. Amer. Math. Soc.}, 35(1):123--209, 2022.

\bibitem{Lichnerowicz1969}
Andr\'{e} Lichnerowicz.
\newblock Applications harmoniques et vari\'{e}t\'{e}s {K}\"{a}hleriennes.
\newblock {\em Rend. Sem. Mat. Fis. Milano}, 39:186--195, 1969.

\bibitem{OO2021}
Yuji Odaka and Yoshiki Oshima.
\newblock {\em Collapsing {K}3 surfaces, tropical geometry and moduli
  compactifications of {S}atake, {M}organ-{S}halen type}, volume~40 of {\em MSJ
  Memoirs}.
\newblock Mathematical Society of Japan, Tokyo, 2021.

\bibitem{Satake1956}
I.~Satake.
\newblock On a generalization of the notion of manifold.
\newblock {\em Proc. Nat. Acad. Sci. U.S.A.}, 42:359--363, 1956.

\bibitem{SZ2022}
Song Sun and Ruobing Zhang.
\newblock Collapsing geometry of hyperk\"ahler 4-manifolds and applications,
  2022.

\bibitem{Yau1978}
Shing~Tung Yau.
\newblock On the {R}icci curvature of a compact {K}\"{a}hler manifold and the
  complex {M}onge-{A}mp\`ere equation. {I}.
\newblock {\em Comm. Pure Appl. Math.}, 31(3):339--411, 1978.

\end{thebibliography}

\end{document}